\documentclass{siamltex}

\usepackage[T1]{fontenc}
\usepackage{lmodern}

\usepackage{microtype}
\usepackage{amsmath,amsfonts,amssymb}
\usepackage{graphicx}

\usepackage{tikz,pgfplots}
\pgfplotsset{compat=newest}
\usetikzlibrary{matrix}

\usepackage[colorinlistoftodos]{todonotes}

\usepackage{bm}

\usepackage{subfig}

\usepackage{marginnote}

\usepackage{fixltx2e}

\usepackage{bbm}
\usepackage[colorlinks=true,linkcolor=blue,citecolor=blue]{hyperref}

\newtheorem{remark}{Remark}
\title{Numerical bifurcation~study of superconducting~patterns on a square}
\author{Nico Schl\"omer\thanks{Departement Wiskunde-Informatica, Universiteit Antwerpen, Middelheimlaan 1, 2020 Antwerpen, Belgium}
\and Daniele Avitabile\thanks{Department of Mathematics, University of Surrey,
Guildford, GU2 7XH, UK}
 \and Wim Vanroose\thanks{Departement Wiskunde-Informatica, Universiteit Antwerpen, Middelheimlaan 1, 2020 Antwerpen, Belgium}}
\newcommand\A{\ensuremath{\mathbf{A}}}
\newcommand\x{\ensuremath{\mathbf{x}}}

\newcommand\B{\ensuremath{\mathbf{B}}}
\renewcommand\H{\ensuremath{\mathbf{H}}}

\newcommand\n{\ensuremath{\mathbf{n}}}
\newcommand\R{\ensuremath{\mathbb{R}}}

\newcommand\K{\ensuremath{\mathbb{K}}}

\newcommand\C{\ensuremath{\mathbb{C}}}
\newcommand\T{\ensuremath{\nxs{S^1}}}
\newcommand\bn{\ensuremath{\bm{\nabla}}}
\newcommand\Dpsi{\ensuremath{\delta\psi}}

\renewcommand\i{\ensuremath{\mathbbm{i}}}
\newcommand\e{\ensuremath{\text{\upshape{e}}}}

\newcommand\GL{\ensuremath{\mathcal{G\!L}}}
\newcommand\dfn{\ensuremath{\mathrel{\mathop:}=}}

\newcommand\igralnl[4]{\ensuremath{\int\nolimits_{#1}^{#2} #3 \, \mathrm{d} #4}}
\newcommand\norm[1]{\ensuremath{\left\| #1 \right\|}}
\newcommand\tp{\ensuremath{\mathrm{T}}}

\newcommand\range{\ensuremath{\mathcal{R}}}
\renewcommand\L{\ensuremath{\mathcal{L}}}
\newcommand\J{\ensuremath{\mathcal{J}}}
\newcommand\Ken{\ensuremath{\mathcal{K}}}
\newcommand\conj[1]{\ensuremath{\overline{#1}}}

\newcommand\relphantom[1]{\mathrel{\phantom{#1}}}

\newcommand{\nxs}{}

\newcommand{\Dfour}{\ensuremath{D_4}}

\newcommand\rev[1]{\textcolor{magenta}{#1}}

\DeclareMathOperator{\spn}{span}

\DeclareMathOperator{\alg}{alg}

\newlength\figurewidth
\newlength\figureheight
\begin{document}

\maketitle
\begin{abstract}
This paper considers the extreme type-II Ginzburg--Landau~equations
that model vortex patterns in superconductors. The nonlinear PDEs are
solved using Newton's method, and properties of the Jacobian operator
are highlighted.  Specifically, it is illustrated how the operator can
be regularized using an appropriate phase condition.  For a
two-dimensional square sample, the numerical results are based on a
finite-difference discretization with link variables that preserves
the gauge invariance. For two exemplary sample sizes, a thorough
bifurcation analysis is performed using the strength of the applied
magnetic field as a bifurcation parameter and focusing on the symmetries of this
system. The analysis gives new insight in the transitions between stable and
unstable states, as well as the connections between stable solution
branches.
\end{abstract}


\section{Introduction}\label{sec:introduction}

In this article, we study the symmetry-breaking transitions between stable and
unstable patterns in small-sized superconducting samples. Superconductors are
materials that expel magnetic fields and exhibit zero electrical resistance
when they are below a characteristic temperature $T_{\mathrm{c}}$.
Mathematically, the superconductor's states are described by a set of
nonlinear PDEs, known as the Ginzburg--Landau system \cite{DGP:1992:AAG}.

For simplicity, let us suppose that a sample of superconducting material occupies an
open, bounded region $\Omega$ of the Euclidean space, immersed in an external
magnetic field $\H_0$ (see Figure~\ref{fig:meissner}). 
Above a critical temperature $T_{\mathrm{c}}$, the material behaves like a normal
conductor:
it exhibits electrical resistivity and is homogeneously penetrated by the applied magnetic field. The material is said
to be in a \emph{homogeneously non-superconducting state} (or \emph{normal
state}).

At low temperatures, $T < T_{\mathrm{c}}$, the material exhibits a \emph{complete}
loss of resistivity, resulting in the formation of superconducting currents in the
sample. Such currents give rise to an induced magnetic field
and the total magnetic field $\B$ is expelled from the interior of the
sample. Below a certain critical field strength $H_{\mathrm{c1}}$, the magnetic field is expelled
entirely; the material is said to be in a \emph{homogeneously superconducting state}.
Below the critical temperature and for
stronger applied magnetic fields, however, mixed configurations can exist:
the magnetic field penetrates only in confined regions of the sample. For so-called
\emph{type-II superconductors}~\cite{goodman1966}, those areas are circular
\emph{vortices}, arranged in characteristic patterns.

\begin{figure}
  \centering
  \includegraphics{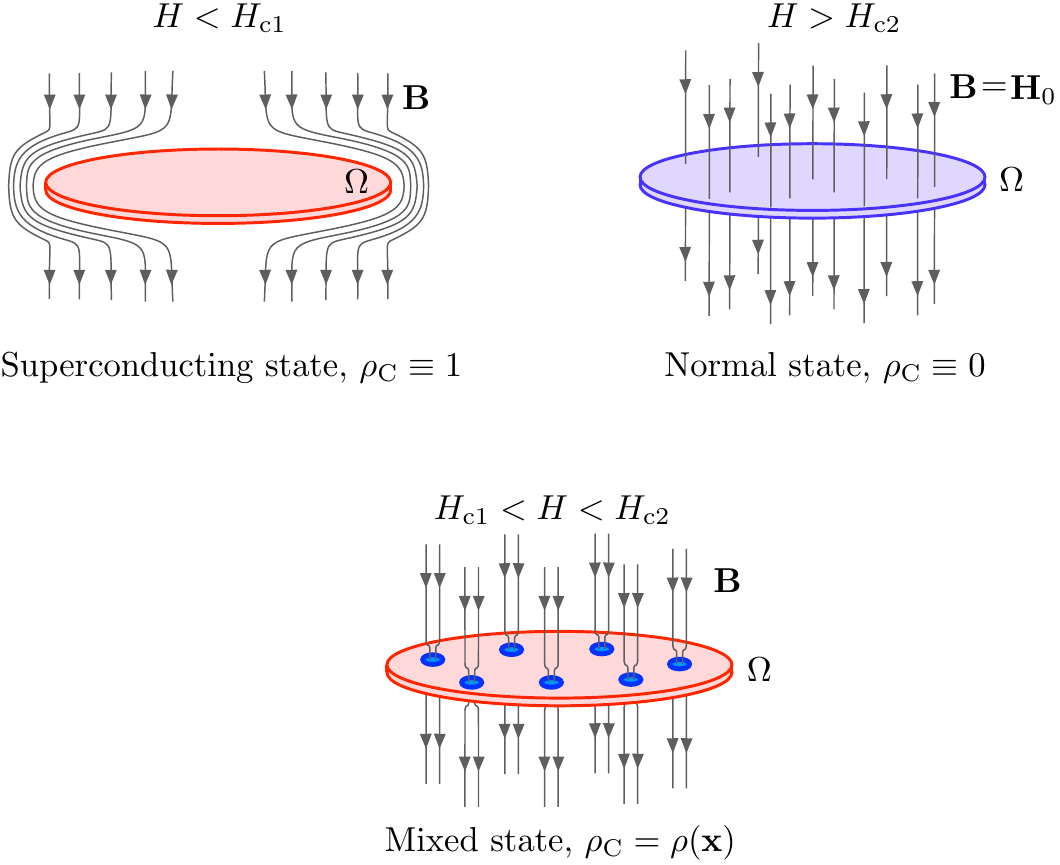}
  \caption{States of a superconducting sample immersed in an external magnetic field $\H_0$.
  Top: below the critical temperature and for $H<H_{\textrm{c}1}$, the sample is in a homogeneously superconducting
  state in which internal currents are generated and the total magnetic field $\B$ is expelled from the
  specimen (left); above the critical temperature or for $H>H_{\textrm{c}2}$, the material is in a normal state and
  the external magnetic field penetrates the whole sample (right). Bottom: type-II
  superconductors can exhibit mixed states, in which vortices of normal
  conductivity are embedded in a superconductive background. In the mixed
  configuration, $\B$ can penetrate the sample only through the 
  vortices, giving rise to characteristic superconductive patterns.}
  \label{fig:meissner}
\end{figure}

In large samples, the vortices organize in a regular pattern, also known as the
\emph{Abrikosov~lattice} (see \cite{abrikosov1957}, \cite{goodman1966} and
references therein). In small samples, however, owing to the boundaries, the observed
patterns can significantly deviate from the regular lattice and their organization
depends sensitively on the intensity of the applied magnetic field as well as
the geometry and the symmetries of the sample.
These small-scale (mesoscopic) systems with simple geometric shapes like discs,
triangles, or squares are of technological interest since they can be
built into nanoscale devices \cite{aladyshkin2009nucleation}.

In applications, one is interested in finding steady states of the system, studying
their stability and their dependence upon the external magnetic field. The state of a
superconducting sample is, in general, characterized by two quantities:
the total magnetic field $\B = \colon \R^3 \to \R^3$ and
the density $\rho_{\text{C}} \colon \Omega \cup \partial \Omega \to \R$ of electron pairs
which constitute superconductivity (Cooper pairs).

A typical approach for studying superconducting states is to define a suitable
Gibbs~energy for the system and to derive a set of evolution equations for the
order parameter $\psi \colon \Omega \cup \partial \Omega \to \C$,
$|\psi|^2=\rho_\textrm{C}$, and
the magnetic vector potential $\A \colon \R^3 \to \R^3$, $\bn\times\A=\B$. The
resulting system in known as the \textit{Ginzburg--Landau system}
\cite{DGP:1992:AAG}. The associated
initial-boundary-value problem has been studied both analytically an numerically.
Various results on the existence and uniqueness of solutions, for example,
can be found in
\cite{bethuel1994ginzburg,sandier2007vortices,LD:1997:GLV} and references therein.

However, it is often necessary to resort to numerical simulation to study the complex
interaction of vortices in samples of arbitrary shapes: a popular strategy is to
time-step the Ginzburg--Landau system via Gauss--Seidel iterations until an
equilibrium is reached; the external magnetic field is then varied quasi-statically,
and a new steady state is found \cite{PhysRevB.65.104515,PhysRevLett.81.2783}. 
We show a typical result of this analysis in Figure~\ref{fig:experimentalCascade}.
The solution branches appear disconnected: when an instability is met, the direct
simulation jumps to a nearby stable branch, as the employed numerical method can
compute only stable solutions. 

\begin{figure}
\centering
\setlength\figurewidth{0.27\textwidth}
\setlength\figureheight{0.9\figurewidth}
\hfill
\subfloat[Disk.]{\begin{tikzpicture}
\begin{axis}[width=\figurewidth,height=\figureheight,scale only axis,xmin=0.0, xmax=2.2, ymin=-1.0, ymax=0.0,ylabel=$F/F_0$]
\addplot [-] table [y expr=\thisrowno{1}*2] {figures/baelus-peeters/enL0.dat};
\addplot [-] table [y expr=\thisrowno{1}*2] {figures/baelus-peeters/enL1.dat};
\addplot [-] table [y expr=\thisrowno{1}*2] {figures/baelus-peeters/enL2.dat};
\addplot [-] table [y expr=\thisrowno{1}*2] {figures/baelus-peeters/enL3.dat};
\addplot [-] table [y expr=\thisrowno{1}*2] {figures/baelus-peeters/enL4.dat};
\addplot [-] table [y expr=\thisrowno{1}*2] {figures/baelus-peeters/enL5.dat};
\addplot [-] table [y expr=\thisrowno{1}*2] {figures/baelus-peeters/enL6.dat};
\addplot [-] table [y expr=\thisrowno{1}*2] {figures/baelus-peeters/enL7.dat};
\addplot [-] table [y expr=\thisrowno{1}*2] {figures/baelus-peeters/enL8.dat};
\addplot [-] table [y expr=\thisrowno{1}*2] {figures/baelus-peeters/enL9.dat};
\addplot [-] table [y expr=\thisrowno{1}*2] {figures/baelus-peeters/enL10.dat};
\addplot [dashed] table [y expr=\thisrowno{1}*2] {figures/baelus-peeters/enL2m.dat};
\addplot [dashed] table [y expr=\thisrowno{1}*2] {figures/baelus-peeters/enL3m.dat};
\addplot [dashed] table [y expr=\thisrowno{1}*2] {figures/baelus-peeters/enL4m.dat};
\addplot [dashed] table [y expr=\thisrowno{1}*2] {figures/baelus-peeters/enL5m.dat};
\end{axis}
\end{tikzpicture}}
\hfill
\subfloat[Square.]{\begin{tikzpicture}
\begin{axis}[width=\figurewidth,height=\figureheight,scale only axis,xmin=0.0, xmax=2.2, ymin=-1.0, ymax=0.0,yticklabels=\empty]
\addplot [-] table [y expr=\thisrowno{1}*2] {figures/baelus-peeters/2enL0.dat};
\addplot [-] table [y expr=\thisrowno{1}*2] {figures/baelus-peeters/2enL1.dat};
\addplot [-] table [y expr=\thisrowno{1}*2] {figures/baelus-peeters/2enL2.dat};
\addplot [-] table [y expr=\thisrowno{1}*2] {figures/baelus-peeters/2enL3.dat};
\addplot [-] table [y expr=\thisrowno{1}*2] {figures/baelus-peeters/2enL6.dat};
\addplot [-] table [y expr=\thisrowno{1}*2] {figures/baelus-peeters/2enL7.dat};
\addplot [-] table [y expr=\thisrowno{1}*2] {figures/baelus-peeters/2enL8.dat};
\addplot [-] table [y expr=\thisrowno{1}*2] {figures/baelus-peeters/2enL9.dat};
\addplot [-] table [y expr=\thisrowno{1}*2] {figures/baelus-peeters/2enL10.dat};
\addplot [-] table [y expr=\thisrowno{1}*2] {figures/baelus-peeters/2enL11.dat};
\addplot [dashed] table [y expr=\thisrowno{1}*2] {figures/baelus-peeters/2enL2m.dat};
\addplot [dashed] table [y expr=\thisrowno{1}*2] {figures/baelus-peeters/2enL3m.dat};
\addplot [dashed] table [y expr=\thisrowno{1}*2] {figures/baelus-peeters/2enL4m.dat};
\addplot [dashed] table [y expr=\thisrowno{1}*2] {figures/baelus-peeters/2enL5m.dat};
\addplot [dashed] table [y expr=\thisrowno{1}*2] {figures/baelus-peeters/2enL6m.dat};
\end{axis}
\end{tikzpicture}}
\hfill
\subfloat[Triangle.]{\begin{tikzpicture}
\begin{axis}[width=\figurewidth,height=\figureheight,scale only axis,xmin=0.0, xmax=2.2, ymin=-1.0, ymax=0.0,yticklabels=\empty]
\addplot [-] table [y expr=\thisrowno{1}*2] {figures/baelus-peeters/3enL0.dat};
\addplot [-] table [y expr=\thisrowno{1}*2] {figures/baelus-peeters/3enL1.dat};
\addplot [-] table [y expr=\thisrowno{1}*2] {figures/baelus-peeters/3enL2.dat};
\addplot [-] table [y expr=\thisrowno{1}*2] {figures/baelus-peeters/3enL3.dat};
\addplot [-] table [y expr=\thisrowno{1}*2] {figures/baelus-peeters/3enL4.dat};
\addplot [-] table [y expr=\thisrowno{1}*2] {figures/baelus-peeters/3enL5.dat};
\addplot [-] table [y expr=\thisrowno{1}*2] {figures/baelus-peeters/3enL7.dat};
\addplot [-] table [y expr=\thisrowno{1}*2] {figures/baelus-peeters/3enL8.dat};
\addplot [-] table [y expr=\thisrowno{1}*2] {figures/baelus-peeters/3enL9.dat};
\addplot [-] table [y expr=\thisrowno{1}*2] {figures/baelus-peeters/3enL10.dat};
\addplot [-] table [y expr=\thisrowno{1}*2] {figures/baelus-peeters/3enL11.dat};
\addplot [dashed] table [y expr=\thisrowno{1}*2] {figures/baelus-peeters/3enL2m.dat};
\addplot [dashed] table [y expr=\thisrowno{1}*2] {figures/baelus-peeters/3enL3m.dat};
\addplot [dashed] table [y expr=\thisrowno{1}*2] {figures/baelus-peeters/3enL4m.dat};
\addplot [dashed] table [y expr=\thisrowno{1}*2] {figures/baelus-peeters/3enL5m.dat};
\addplot [dashed] table [y expr=\thisrowno{1}*2] {figures/baelus-peeters/3enL6m.dat};
\end{axis}
\end{tikzpicture}}
\hfill
  \caption{(Reproduced from \cite{PhysRevB.65.104515}.)
Typical cascades of branches of steady states of the Ginzburg--Landau problem
for various two-dimensional sample shapes.
The plots show the strength of the applied magnetic field (homogeneous and perpendicular
to the sample) versus the normalized energy of the states.
}
\label{fig:experimentalCascade}
\end{figure}
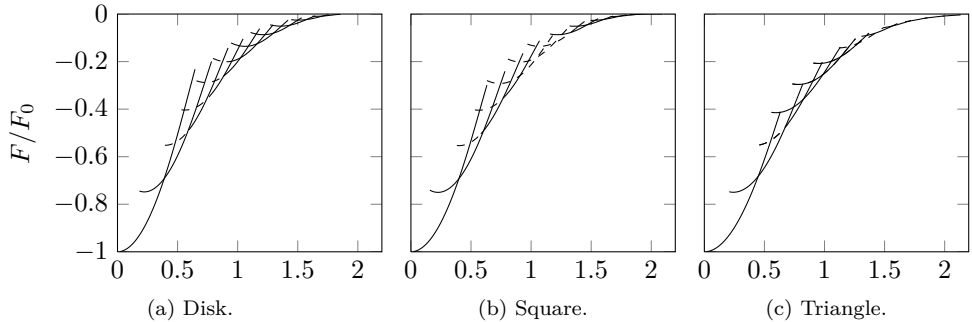

The plots in Figure~\ref{fig:experimentalCascade} are in good agreement with the
hysteretic behavior that has also been observed
experimentally~\cite{Schweitzer1967}, but they are not yet fully understood
from the point of view of bifurcation analysis. The main results in this direction
are confined to one-dimensional spatial domains (see \cite{dancer2000global,
aftaliontroy, aftalion2000asymptotic} and references therein). The two-dimensional
case has been studied by means of direct numerical
simulation for various material parameters and strengths of the applied magnetic
field~\cite{aftalion2002bifurcation}, as well as for a variety of different shapes
and domain sizes (see
\cite{PhysRevLett.79.4653,PhysRevB.65.104515,PhysRevB.70.144523} and references
therein), but the bifurcation scenario of the Ginzburg--Landau problem in two and
three dimensions is largely unexplored.

The main motivation of the present paper is to classify the instabilities occurring
in superconducting samples using numerical continuation, as opposed to time-dependent
simulations. To this end, we define a well-posed boundary-value problem, choose a
spatial discretization, and find steady states of the Ginzburg--Landau problem by
Newton iterations. 
More specifically, we focus on square samples of extreme type-II
superconductors subject to a homogeneous external magnetic field. In this case, the
Ginzburg--Landau problem simplifies considerably as it is possible to derive the
vector potential $\A$ explicitly and then solve a nonlinear partial
differential equation for the order parameter $\psi$.

We expect that the symmetries of the problem will influence the
bifurcation landscape. As we will see, the relevant groups for the computations
presented in this paper are the circle group $\T$ and the dihedral group \Dfour{}.
The discrete \Dfour{} symmetry suggests that we can use
the \emph{equivariant branching lemma} to predict symmetries of the emerging
branches at bifurcation points \cite{golubitsky2002symmetry,Hoyle:2006:PF}. 
On the other hand, the continuous $\T$-symmetry
induces the presence of a zero eigenvalue in the spectrum of the linear operator
associated with the boundary-value formulation, causing problems to the
convergence of the Newton iterations.

We regularize the system by extending the boundary-value problem and employing a
suitable phase condition, using the framework proposed by Champneys and Sandstede
\cite{CS:2007:NCC}. The extended boundary-value formulation is then discretized using
a common gauge-preserving technique and the patterns are path-followed in parameter
space via pseudo-arclength continuation.
 
To the best of our knowledge, this approach has never been employed before for the
Ginzburg--Landau problem, albeit the application of Newton's method has been proposed
in \cite{DGP:1992:AAG} and inexact Newton methods are often used in
practice~\cite{KK:1995:VCT}. In addition, equivariant bifurcation theory has never
been used to explain the instabilities found experimentally and numerically in
superconducting samples, even though the importance of symmetries was pointed out in
\cite{chibotaru2000symmetry}, where the system is linearized around the trivial
steady state and the relative eigenmodes are studied in the context of
$C_4$-symmetries.

The main result of the present paper is a classification of the
bifurcations occurring in square domains of small and moderate
sizes. In small samples, where the domain can host just a single
superconducting vortex, the bifurcations are entirely determined by
the natural two-dimensional irreducible representation of $D_4$ (see
\cite{Hoyle:2006:PF}, Section $4.3$). However, as the domain size
increases, the bifurcation diagram gets more complicated and it
involves also one-dimensional irreducible representations of
$D_4$. Furthermore, in larger samples we compute stable vortices of
higher multiplicity: these structures were previously found by direct
numerical simulation \cite{PhysRevB.65.104515}, but their formation
was still an open problem; our analysis shows that vortices with
different multiplicity are all linked in parameter space via
symmetry-breaking bifurcations. Furthermore, we used Newton-Krylov methods to solve
the system, exploiting the properties of the Jacobian operator in the Krylov
iterations.

The remainder of the article is organized as follows. Section
\ref{sec:review} discusses the Ginzburg--Landau system in the
large-$\kappa$ limit and details its symmetries.
Section~\ref{sec:self adjointness} contains material on the
  linearization of the Ginzburg--Landau system and its
  self-adjointness, which is of importance for the numerical solution
  of the linear system associated with each Newton iteration.
Section~\ref{subsection:regularization} is concerned with the
regularization of the equations. Details on the discretization of the
system with link variables and properties thereof can be found in
Section~\ref{sec:discretized}. The numerical computations are included in
Section~\ref{sec:results}, where we show families of solutions as a
function of the strength of the applied magnetic field. We relate the
bifurcations to the symmetries of the solutions with the help of the
equivariant branching lemma.  The appendix contains an extension of
Keller's bordering lemma which is used in
Sections~\ref{subsection:regularization} and \ref{sec:discretized}.

\paragraph{Notations}
Throughout this article we will use bold symbols ($\x$) for vector-valued quantities.
For any $z\in\C$, $\Re(z)$ and $\Im(z)$ denote its real and imaginary parts, 
$\overline{z}$ is used for complex conjugation. Similarly, for $\psi:\Omega\to\C$,
$\overline{\psi}:\Omega\to\C$ is such that $\overline{\psi}(x)\dfn
\overline{\psi(x)}$ for all $x\in\Omega$. For function spaces, we use
$C^{k}(\Omega)$ denote the vector space of all $k$
times differentiable functions.
We use $L^2_{\K}(\Omega)$ for the Hilbert-space in the field $\K$
of square-integrable functions
over $\Omega$, equipped with the inner product $\left\langle
\varphi,\psi\right\rangle = \int_\Omega \overline{\varphi} \psi\,\mathrm{d}\Omega$ for all
$\varphi,\psi\in L^2_{\K}(\Omega)$. 
The range of a linear operator $L$ is denoted by $\range(L)$.
For the symmetry groups under consideration, the symbol $\T$ is used
to denote the circle group $\{z \in \mathbb \C : |z| = 1\}$ (which is group-isomorphic
to SO(2)).
For a given state $\psi$, $\Sigma_{\psi}$ denotes the symmetry group under the
action of which $\psi$ is invariant.

\section{The Ginzburg--Landau equation}\label{sec:review}

For an open, bounded domain $\Omega\subset\nxs{\R^3}$,
with a piecewise smooth boundary $\partial\Omega$, the Ginzburg--Landau
problem is usually derived by minimizing the Gibbs free
energy functional 
\begin{equation}\label
{eq:Gibbs}
\begin{split}
G(\psi,\A) - G_{\mathrm{n}}
&=
\xi\frac{|\alpha|^2}{\beta}
\int_{\Omega}
\Bigg[
-|\psi|^2
+ \frac{1}{2}|\psi|^4
+ \left|-\i\bn\psi - \A\psi \right|^2\\
&\relphantom{=}\phantom{\xi\frac{|\alpha|^2}{\beta}\int_{\Omega}\Bigg[}
+ \kappa^2 (\bn\times \A)^2
- 2\kappa^2 (\bn\times \A)\cdot \H_0
\Bigg]\, \mathrm{d}\Omega,
\end{split}
\end{equation}
where the state $(\psi,\A)$ is in the natural energy space
such that the integral is well-defined \cite{DGP:1992:AAG}.
As we have seen in the introduction, the scalar $\psi$ is commonly
referred to as the \emph{order parameter}, while $\A$ is
the magnetic vector potential corresponding to the total magnetic field.
The physical observables associated with the state $(\psi,\A)$ are the density
$\rho_{\text{C}}=|\psi|^2$ of the superconducting charge carriers
and the total magnetic
field $\B=\bn\times\A$. The constant $G_{\mathrm{n}}$ represents the energy
associated with the entirely normal (non-superconducting) state.

The energy~(\ref{eq:Gibbs}) is written in its dimensionless form and it depends upon
the impinging magnetic field $\H_0$ and the material parameters $\alpha, \beta,
\kappa, \xi\in\R$. The most relevant parameters are $\kappa$ and
$\xi$; in particular, $\kappa=\lambda/\xi$ is the ratio of the penetration depth $\lambda$
(the length scale at which the magnetic field penetrates the sample) to the coherence
length $\xi$ (the characteristic spatial scale of $\psi$). A superconductor is said
to be of \textit{type I} if $\kappa < 1/\sqrt{2}$, and of \textit{type II} otherwise.

To complete our description of the Gibbs energy, we remark that we have scaled
the domain $\Omega$ in units of the coherence length $\xi$ while another common
choice is to scale the domain by $\lambda$ \cite{DGP:1992:AAG}.

Starting from the Gibbs energy and using standard calculus of variations, it is
possible to derive the Ginzburg--Landau equations \cite{DGP:1992:AAG}, a
boundary-value problem in the unknowns $\psi$ and $\A$.
As anticipated in the introduction, we will simplify the problem
and consider only the limit $\kappa\to\infty$ (\textit{extreme type-II
superconductors}): this approximation gives satisfactory results for all
high-temperature superconductors with large but finite values of $\kappa$, typically
$50<\kappa<100$.

In this case, the Ginzburg--Landau problem decouples and we have
\begin{equation}\label{eq:GL}
\begin{cases}
0 = \left(-\i\bn - \A\right)^2 \psi - \psi \left(1 - |\psi|^2\right) \quad \text{in } \Omega,  \\[3mm]
0 = \n \cdot ( -\i\bn - \A) \psi \quad \text{on } \partial\Omega,
\end{cases}
\end{equation}
where $\A=\A(\H_0)$ is given by the relations
\begin{equation}\label{eq:GLA}
\begin{cases}
\bn\times(\bn\times \A) = 0 \quad \text{in } \Omega,\\
\n \times (\bn\times\A) = \n \times  \H_0\text{ on } \partial\Omega.
\end{cases}
\end{equation}
Since for this decoupled system there are no magnetization effects,
the magnetic fields $\H_0$ and $\B$ coincide.

Since the sample's width scales with $\xi$, the large-$\kappa$ limit
$\lambda \gg \xi$ means that the magnetic field $\A$
penetrates the whole sample, independently of $\psi$.

In passing, we note that Equation~\eqref{eq:GL} does not
coincide with the so-called \emph{Complex Ginzburg--Landau equation} (see
\cite{aranson2002} and references therein).

In the present paper, we consider a two-dimensional square sample
\[
\Omega=\Omega_d\dfn\{ (x,y,z)\in\R^3: (x,y)\in(-d/2,d/2)^2, z=0 \}, \quad d\in\R^+,
\]
subject to a perpendicular, homogeneous magnetic field $\H_0 = (0,0,\mu)^\tp$, $\mu\in\R$.
From \eqref{eq:GLA} we can derive an expression for the induced vector potential
\begin{equation}\label{eq:Amu}
\A(x,y;\mu) = (A_x(x,y;\mu), A_y(x,y;\mu) )^\tp \dfn \frac{1}{2}( - \mu y, \mu x )^\tp,
\end{equation}
where we have deliberately omitted the third component.  

In conclusion, we will consider the following boundary-value problem with $X_d$ being
the natural energy space over $\Omega_d$ associated
with the Gibbs~energy \eqref{eq:Gibbs} and $Y_d$ its dual space.
The equations are
\begin{equation}\label{eq:GL compact}
\begin{split}
&\relphantom{=} \GL(\psi; \mu): X_d\times\R \to Y_d,\\[0.5ex]
0 &= \GL(\psi; \mu) \dfn
\begin{cases}
\left(-\i\bn - \A(\mu)\right)^2 \psi - \psi \left(1 - |\psi|^2\right) \quad \text{on } \Omega_d,\\[0.5ex]
\n \cdot ( -\i\bn - \A(\mu)) \psi \quad \text{on } \partial\Omega_d,
\end{cases}
\end{split}
\end{equation}
where $\A(\mu)$ is given by \eqref{eq:Amu}, with the parameters $\mu\in\R$, $d\in\R^+$ and
with the boundary conditions given in the sense of traces. To shorten the notation,
the dependence of $\A$ on $\mu$ will often not made explicit in the remainder of the
text.

Note that, because $\Omega_d$ is convex and
$\A\in C^{\infty}(\overline{\Omega})$,
solutions in the natural energy space immediately have higher regularity \cite{BBX:2003:REE}
and in fact coincide with the classical strong solutions in $C^2(\Omega)\cap C^1(\overline{\Omega})$.

\paragraph{Symmetries}\label{sec:symmetries}
As mentioned in Section~\ref{sec:introduction}, symmetries play an important role in
the bifurcations scenario of our problem.
The Ginzburg--Landau system for extreme type-II superconductors, \eqref{eq:GL compact},
is left invariant by the action
of the circle group $\T$,
\begin{equation}\label{eq:reduced gauge invariance}
\theta_{\eta} \colon \psi \longmapsto \psi \exp(\i\eta), \quad \eta\in[0,2\pi).
\end{equation}
The circle-group symmetry is also referred to as \emph{phase symmetry}.

\begin{figure}
\centering
\includegraphics[width=0.85\textwidth]{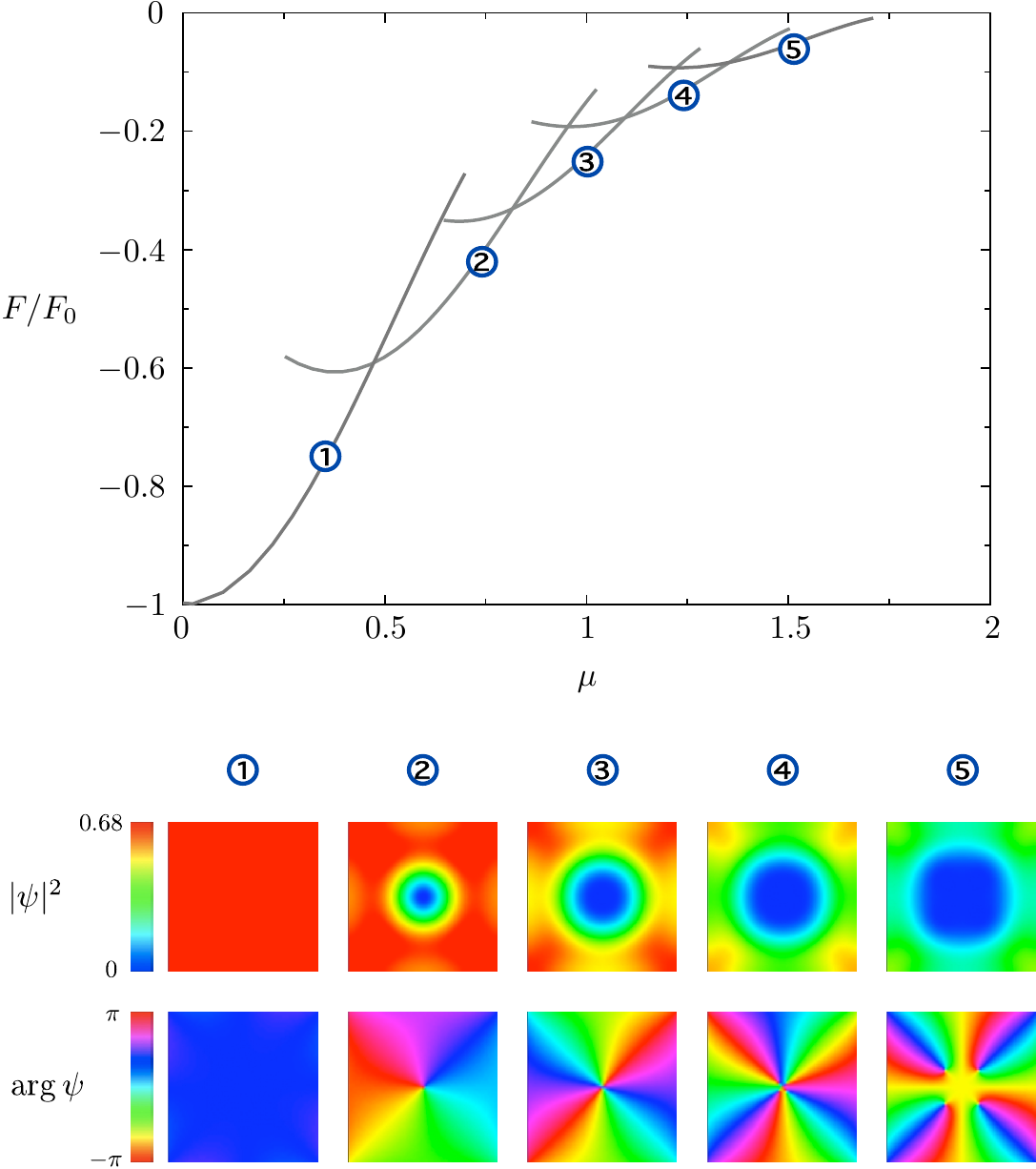}
\caption{Several families of stable solution branches of the Ginzburg--Landau
equation \eqref{eq:GL compact} as a function of the bifurcation parameter $\mu$,
the intensity of the applied magnetic field. The results, obtained via numerical
continuation, are computed for a square of side length $d = 5.5$.  
Top: On the vertical axis, we plot a measure of the Gibbs energy of the states (see
remark~\ref{remark:energy} on page~\pageref{remark:energy}).
We observe a cascade of instabilities similar to the one shown in the center panel of
Figure~\ref{fig:experimentalCascade}, albeit the number of solution branches is
larger in the latter diagram.
As we will see in Section~\ref{sec:intermediate}, the number of stable
primary branches increases with the domain size $d$ (the computations in the center
panel of Figure~\ref{fig:experimentalCascade} are for $d=7.1$). 
Bottom: Selected stable patterns along the
branches. Vortices are characterized by a localized region of low
supercurrent density (blue in the pictures), as well as a $2\pi
k$-phase change in $\arg\psi$, where $k$ is the \emph{multiplicity} of
the vortex. For higher values of $\mu$, vortices with higher
multiplicities are found: pattern $3$, for instance, has a $2 \times
2\pi$ phase change, it is a \emph{giant vortex} with multiplicity
$2$. All patterns shown have full $D_4$-symmetry and we expect to
interpret the diagram in terms of symmetry-breaking bifurcations.}
\label{fig:example solutions}
\end{figure}

In addition, $\GL(\psi;\mu)$ is invariant under rotations by $\pi/2$
\begin{align}
  \rho & \colon \psi(x,y) \longmapsto \psi(-y,x) \label{eq:rotation} \\
  \intertext{and conjugated mirroring along the $y$-axis,}
  \sigma & \colon \psi(x,y) \longmapsto \conj{\psi(-x,y)}\label{eq:mirror}.
\end{align}
Note that, up to conjugation in $\sigma$, these are the classical group actions that
generate the $D_4$ symmetry group of the square. In fact, the group generated by
$\rho$ and $\sigma$ is isomorphic to $D_4$.
Even though symmetries of the Ginzburg--Landau problem have been considered before
\cite{chibotaru2000symmetry}, the analysis was limited only to a linearization of the
Ginzburg--Landau operator in the presence of rotations~\eqref{eq:rotation}; in our
case, we will consider the nonlinear problem and account also for 
conjugate reflections~\eqref{eq:mirror}.

In conclusion, the relevant symmetry group for our problem is generated by the actions
\eqref{eq:reduced gauge invariance}--\eqref{eq:mirror},
\[
\Gamma \dfn \langle \theta_\eta, \rho, \sigma \rangle \cong \T \times D_4. 
\]

We refer to the reader to Section~\ref{subsection:regularization}, where we will
explain how to factor out the continuous $\T$-symmetry that induces a singularity in
the boundary-value problem associated with $\GL(\psi;\mu)$, and we conclude this
section by showing in Figure~\ref{fig:example solutions} a few examples of patterns computed
via numerical continuation. The Ginzburg--Landau problem~\eqref{eq:GL compact}
possesses two trivial solutions:  $\GL(0;\mu) = 0$ for
all $\mu\in\R$ (the normal state) and $\GL(1;0) = 0$ (the homogeneously superconducting
state). As expected, we find branches of nontrivial stable $D_4$-symmetric solutions
arranged in a characteristic cascade, in agreement with the results obtained by
direct numerical simulations where the parameter $\mu$ is varied
quasi-statically (see Figure~\ref{fig:experimentalCascade}).

\section{The Jacobian operator}\label{sec:self adjointness}
The patterns shown in Figure~\ref{fig:example solutions} were computed as regular
zeros of a nonlinear system of equations derived from the Ginzburg--Landau problem
\eqref{eq:GL compact}. The solutions were found via Newton-Krylov
iterations, that require the specification of the action of the Jacobian associated
with $\GL(\psi;\mu)$ (see \cite{Kelley1995} for details on iterative linear solvers).
Even though there were previous attempts to
solve \eqref{eq:GL compact}
with a modified Newton's method \cite{KK:1995:VCT}, those implementations did not
retain second-order convergence.
Before deriving explicitly the regularization procedure that allowed us to compute
the superconducting patterns, we introduce in this section the Jacobian
operator $\J(\psi;\mu)$ associated with $\GL(\psi;\mu)$, and prove its
self-adjointness with respect to a suitably-defined inner product in $X_d$.

For a given
$d\in\R^+$, $\mu\in\R$, and $\psi,\delta\psi\in X_d$, let us consider
\[
\begin{split}
 & \GL(\psi+\Dpsi;\mu) - \GL(\psi;\mu)\\
&= \left[(-\i \bn - \A)^2(\psi + \Dpsi)
- (\psi + \Dpsi)\left(1 - \conj{(\psi +\Dpsi)}(\psi + \Dpsi)\right) \right]\\
&\relphantom{=}{} - \left[(-\i \bn - \A)^2 \psi - \psi\left(1 - \conj{\psi}\psi\right) \right]\\
&            =           (-\i \bn - \A)^2\Dpsi    + \psi\left( \conj{\psi}\,\Dpsi+\psi{}\,\conj{\Dpsi} + \conj{\Dpsi}\,\Dpsi \right)\\
&\relphantom{=} \phantom{(-\i \bn - \A)^2\Dpsi}{} - \Dpsi \left(1-\conj{\psi}\psi\right)\\
&\relphantom{=} \phantom{(-\i \bn - \A)^2\Dpsi}{} + \Dpsi\left( \conj{\psi}\,\Dpsi+\psi\,\conj{\Dpsi} + \conj{\Dpsi}\,\Dpsi \right).
\end{split}
\]
Neglecting higher-order terms in $\Dpsi$, we obtain the Jacobian operator
\begin{equation}\label{eq:jacobian}
\begin{split}
&\J(\psi;\mu): X_d \to Y_d,\\
&\J(\psi;\mu)\varphi
\dfn \left( (-\i \bn - \A)^2  - 1 + 2|\psi|^2 \right) \varphi + \psi^2 \conj{\varphi}.
\end{split}
\end{equation}
Note that $\J(\psi;\mu)$ is indeed linear when defined over $X_d$ and $Y_d$ as $\R$-vector spaces.

We are now going to prove that the Jacobian operator~(\ref{eq:jacobian}) is self-adjoint with respect
to the inner product
\begin{equation}\label{eq:real inner product}
\left\langle\cdot,\cdot\right\rangle_{\R}\dfn \Re\left\langle\cdot,\cdot\right\rangle_{L^2_{\C}(\Omega)}.
\end{equation}
This property allows us to employ
standard methods for symmetric linear systems such as the
conjugate gradient or the minimal residual method (using this inner product) 
to invert the Jacobian at each Newton iteration.
Note that $\left\langle\cdot,\cdot\right\rangle_{\R}$
coincides with the natural inner product in $(L^2_{\R}(\Omega))^2$,
which is isomorphic to $L^2_{\C}(\Omega)$, because for any given pair
$\phi, \psi\in L^2_{\C}(\Omega)$, one has
\[
\left\langle
\begin{pmatrix}
\Re\phi\\
\Im\phi\\
\end{pmatrix},
\begin{pmatrix}
\Re\psi\\
\Im\psi\\
\end{pmatrix}
\right\rangle_{(L^2_{\R}(\Omega))^2}
=
\left\langle
\Re\phi, \Re\psi
\right\rangle_{L^2_{\R}(\Omega)}
+
\left\langle
\Im\phi, \Im\psi
\right\rangle_{L^2_{\R}(\Omega)}
=
\left\langle
\phi, \psi
\right\rangle_{\R}.
\]

The following lemma gives insight into the adjoint of a linear operator
of the form~(\ref{eq:jacobian}). The lemma is formulated for general
Hilbert spaces, and we will use it for $H=L^2_{\C}(\Omega)$; in our case, the
operation $C$ mentioned below will be the pointwise complex conjugation.

\begin{lemma}\label{lemma: hilbert space adjointness}
Let $H$ be a Hilbert-space with inner product $\left\langle\cdot,\cdot\right\rangle_H$
and let there be an operation $C:H\to H$ such that
\begin{align}\label{eq:theta}
  & C(\alpha x) = \alpha C(x)\quad \textrm{for all $\alpha\in\R, x\in H$}\\
  & \Re\left\langle C(x), y\right\rangle_H = \Re\left\langle x,
  C(y)\right\rangle_H \quad \textrm{for all $x,y\in H$}.
\end{align}
Let $\L_1,\L_2:H\to H$ be linear operators. For every $x\in H$, let
\[
\L x \dfn \L_1 x + \L_2 C(x).
\]
Then $\L$ is a linear operator on $H$ as $\R$-vector space, and its
adjoint $\L^*$ with respect to the inner product~$\left\langle\cdot,\cdot\right\rangle_{\R}\dfn \Re\left\langle\cdot,\cdot\right\rangle_H$
is given by
\[
\L^* x \dfn \L_1^* x + C(\L_2^* x),
\]
where $\L_1^*$, $\L_2^*$ are the adjoint operators in $H$ of $\L_1$, $\L_2$,
respectively.
\end{lemma}
\begin{proof}
Let $x,y\in H$, and consider
\[
\left\langle x , \L y\right\rangle_{\R}
= \Re\left\langle x , \L_1 \nxs{y} + \L_2 C(y)\right\rangle_H
= \Re\left\langle x,\L_1 y \right\rangle_H + \Re\left\langle x, \L_2 C(y) \right\rangle_H.
\]
Using the operator adjoints $\L_1^*$, $\L_2^*$, we get
\begin{multline*}
\left\langle x , \L y\right\rangle_{\R}
= \Re\left\langle\L_1^*x, y \right\rangle_H + \Re\left\langle \L_2^*x,  C(y) \right\rangle\\
= \Re\left\langle\L_1^*x, y \right\rangle_H + \Re\left\langle C(\L_2^*x),  y \right\rangle
= \left\langle\L_1^*x +C(\L_2^*x) , y \right\rangle_{\R}
= \left\langle \L^* x , y\right\rangle_{\R}.
\end{multline*}
\end{proof}

\begin{lemma}\label{lemma: linear part adjoint}
Let $\A(\mu) \in C^1_{\R^n}(\Omega)$, $n\in\{2,3\}$. The kinetic energy operator
\begin{equation}\label{eq:kinetic energy operator}
\begin{split}
&\Ken(\mu): X_d \to Y_d,\\
&\Ken(\mu)\varphi \dfn
\begin{cases}
(-\i \bn - \A)^2 \varphi,\;&\text{in } \Omega_d\\
\n\cdot(-\i \bn - \A)\varphi,\;&\text{on } \partial\Omega_d\\
\end{cases}
\end{split}
\end{equation}
is self-adjoint with respect to the inner product
$\left\langle\cdot,\cdot\right\rangle_{L^2_{\C}(\Omega_d)}$
\nxs{over the subspace $X_d^0\subseteq X_d$ with $\n\cdot(-\i \bn - \A)\varphi=0$}.
\end{lemma}
\begin{proof}
This result immediately derives from the fact that
\begin{equation}\label{eq:keo property}
\int_{\Omega} \overline{\psi} (-\i\bn-\A)^2\varphi \,\mathrm{d}\Omega
= \int_{\Omega} \overline{(-\i\bn-\A)\psi} (-\i\bn-\A)\varphi\,\mathrm{d}\Omega
\nxs{-\i\int_{\partial\Omega} \overline{\psi}\n\cdot(-\i\bn-\A)\varphi}
\end{equation}
for all $\psi\in L^2_{\C}(\Omega_d)$, $\varphi\in X_d$, see \cite{Avron1978}.
\end{proof}

\begin{corollary}\label{corollary:j self-adjoint}
  For any $\psi\in X_d$ and $\A(\mu)\in C^1_{\R^n}(\Omega)$, the Jacobian operator
  $\J(\psi;\mu)$ defined in \eqref{eq:jacobian} is linear and self-adjoint \nxs{over $X_d^0$} with
  respect to the inner product \eqref{eq:real inner product},
  $\langle\cdot,\cdot\rangle_{\R}$.
\end{corollary}
\begin{proof}
  By lemma~\ref{lemma: linear part adjoint}, the operator of
  $\J_1(\psi;\mu)\dfn\left( (-\i \bn - \A)^2 - 1 + 2|\psi|^2 \right)$
  defined over \nxs{$X_d^0$} with respect to the $L^2_{\C}(\Omega_d)$-inner product
  is self-adjoint.  It can easily be checked that the adjoint operator
  of $\J_2(\psi)$, defined by $\J_2(\psi)\varphi = \psi^2\varphi$ for
  $\varphi\in X_d$, is given by $\J^*_2(\psi)\varphi =
  \overline{\psi^2}\varphi$.  Also note that the complex conjugation
  fulfills the conditions (\ref{eq:theta}).  Application of
  lemma~\ref{lemma: hilbert space adjointness} then states that the
  adjoint of $\J(\psi;\mu):X_d\to Y_d$ is given by
\[
\J^*(\psi;\mu) \varphi =
\J^*_1\varphi + \overline{\overline{\psi^2}\varphi}
= \J_1\varphi + \psi^2\overline{\varphi}
= \J(\psi;\mu) \varphi
\]
for all $\varphi\in X_d$, and thus $\J^*(\psi;\mu) = \J(\psi;\mu)$.
\end{proof}

\section{Nullspace and regularization with a phase condition}\label{subsection:regularization}
As stated in the previous sections, our aim is to compute solutions to the
Ginzburg--Landau problem~\eqref{eq:GL compact} and continue
them in the parameter $\mu$. This can be done in principle by discretizing the
Ginzburg--Landau operator and applying standard numerical continuation techniques.

However, as we have seen in Section~\ref{sec:review}, the boundary-value problem
$\GL(\psi;\mu)=0$ is invariant under the actions of the group $\Gamma = \T
\times D_4$, and continuous symmetries (such as the phase symmetry determined by the circle group $\T$)
make the problem ill-posed. After discretization, this leads to numerical difficulties
that make it principally impossible to compute accurate approximations to the original
problem~\cite{demmel} (see Figure~\ref{fig:Newton without bordering}).

This problem is usually met in computations of relative equilibria, which are time-dependent solutions
whose temporal evolution is governed by a symmetry of the underlying differential
equations. Typical examples are traveling waves (translational symmetries) and
spiral waves (rotational symmetry).
The continuous symmetry induces a zero eigenvalue in the Jacobian associated with the
boundary-value problem, and it is therefore not straightforward to use Newton's method to compute
the desired pattern: each Newton iteration inverts the Jacobian
evaluated at a given solution $\psi$ and requires a regular linear operator.

A generic strategy to compute relative equilibria and to remove the
  singularity is to extend the boundary-value problem by introducing
an additional scalar unknown and closing the system by means of a
suitably-defined \emph{phase condition}
\cite{beyn2004,CS:2007:NCC,rowley2003,beyn2007}. The new
boundary-value problem is well-posed, and therefore Newton's method can
find the solution and path-follow it as a function of the
parameters.  This regularization technique can be applied to the
stationary patterns of the Ginzburg--Landau problem to factor out the
action of the circle group $\T$.
To the authors' knowledge, the removal of the singularity for the Ginzburg--Landau
system has not been considered in literature before.  The regularization adopted here
is an application of the framework proposed
in \cite{CS:2007:NCC}.

In order to regularize the Ginzburg--Landau problem, we look at the action of
$\xi_{\eta}\in\alg(\T)$, $\eta\in\R$, on a state $\psi$. Note that the exponential
map of $\T$ is given by $\exp: \alg(\T) \rightarrow \T$, $\xi_\eta \mapsto \theta_\eta$,
where $\theta_\eta$ is identified with the action $\theta_\eta \psi = \e^{\i\eta} \psi$
on a state $\psi$.
This yields
\[
\xi_\eta \psi
= \left.\frac{\text{d}}{\text{d}t} \exp(\xi_{\eta t})\psi\right|_{t=0}
= \left.\frac{\text{d}}{\text{d}t} \theta_{\eta t}\psi\right|_{t=0}
= \left.\frac{\text{d}}{\text{d}t} \e^{\i\eta t}\psi\right|_{t=0} =  \i \eta\,\psi,
\]
and indeed the function $\i\psi_{\text{s}}$ is
in the nullspace of the Jacobian for a solution $(\psi_{\text{s}},\mu_{\text{s}})$ of
(\ref{eq:GL compact}):
\begin{multline}\label{eq:jacobian nullspace}
\J(\psi_{\text{s}};\mu_{\text{s}}) (\i\psi_{\text{s}})
=  \left[ (-\i \bn - \A)^2  - 1 + 2|\psi_{\text{s}}|^2 \right] (\i\psi_{\text{s}}) - \i \psi_{\text{s}}^2  \conj{\psi_{\text{s}}}\\
=  \left(1-|\psi_{\text{s}}|^2\right)(\i\psi_{\text{s}})  - \i\psi_{\text{s}} + 2\i\conj{\psi_{\text{s}}}\psi_{\text{s}}^2  - \i \psi_{\text{s}}^2  \conj{\psi_{\text{s}}}
= 0.
\end{multline}

It is then possible to amend the Ginzburg--Landau problem and factor
out the action of $\T$.
To this end, for fixed $\mu$, we compute $(\psi,\eta)$ as a regular zero of the
extended operator
\[
\begin{split}
\GL_p \colon X_d \times \alg(\T) &\longrightarrow Y_d \times \R\\
                     (\psi,\eta) &\longmapsto ( \GL(\psi;\mu) - \xi_\eta \psi,\, \Phi(\psi -\psi_0) )
\end{split}
\]
where $\Phi \colon X_d \to \R$ is a suitable phase condition and
$\psi_0$ a given reference state.
In the Ginzburg--Landau setting, the natural choice is the functional
\begin{equation}\label{eq:phasecondition}
  \Phi \colon X_d \longrightarrow \R, \quad \psi \longmapsto \Re\langle \i \psi_0, \psi-\psi_0 \rangle
\end{equation}
with a given reference state $\psi_0\in L^2_{\C}(\Omega_d)$, subject to mild conditions
(see corollary~\ref{cor:nullspace reduction}). Hence, instead of
(\ref{eq:GL compact}), we will consider the extended problem
\begin{equation}\label{eq:extended GL}
0 = \GL_p(\psi,\eta; \mu)
\dfn
\begin{pmatrix}
\GL( \psi;  \mu ) - \i\eta \psi\\[0.5ex]
\Im(\left\langle\psi_0,\psi\right\rangle_{L^2_{\C}(\Omega_d)})
\end{pmatrix}.
\end{equation}
\begin{remark}
The phase condition featuring in~\eqref{eq:extended GL} is also a necessary condition for
\[
\min\limits_{\chi\in\R} \left\|\psi_0-\psi\,\e^{\i\chi}\right\|_{L^2_{\C}(\Omega_d)}^2.
\]
This selects, out of all physically equivalent candidate solution
states $\psi\,\e^{\i\chi}$, those two which are closest and furthest from $\psi_0$
in the $L^2_{\C}(\Omega_d)$-norm.
\end{remark}

If $\psi_\text{s}\in X_d$ is a solution of the original
equations~(\ref{eq:GL compact}), then $(\exp(\i\chi_s)\psi_\text{s},0)^\tp$
with $\chi_s\dfn -\arg(\langle\psi_0,\psi_{\text{s}}\rangle_{L^2_{\C}(\Omega_d)})$ is a
solution of~(\ref{eq:extended GL}) as well.
The Jacobian operator corresponding to the extended problem~(\ref{eq:extended GL}) is
\begin{equation}\label{eq:extended jacobian}
\begin{split}
&\J_p(\psi,\eta; \mu) : X_d\times\R \to Y_d\times\R,\\
&\J_p(\psi,\eta; \mu)\,
\begin{pmatrix}
\varphi\\
\nu
\end{pmatrix}
=
\begin{pmatrix}
(\J(\psi;\mu) - \i\eta)\varphi - \i \psi \nu\\[0.5ex]
\Im(\left\langle\psi_0,\varphi\right\rangle_{L^2_{\C}(\Omega_d)})
\end{pmatrix}.
\end{split}
\end{equation}
We expect that the dimension of the nullspace of the extended
Jacobian~(\ref{eq:extended jacobian}) is lower than the one of
$\mathcal{J}(\psi;\mu)$. This is guaranteed by Keller's bordering
lemma~\cite{Keller:1976:NSB} if $\dim \ker \mathcal{J}(\psi;\mu)=1$.
However, in the case of the Ginzburg--Landau operator we will encounter degeneracies
of higher order. In the appendix, we present a bordering lemma that can be applied in such
cases (Lemma~\ref{lemma:keller}) and that is used in the proof of the following
corollary.

\begin{corollary}\label{cor:nullspace reduction}
Let $(\psi_{\text{s}},\mu_{\text{s}})\in X_d\times\R$ be a solution of the
original Ginzburg--Landau equations~(\ref{eq:GL compact}) with $\psi_{\text{s}}\neq 0$, and let $\psi_0\in L^2_{\C}(\Omega_d)$
such that $\left\langle\psi_{0},\psi_{\text{s}}\right\rangle_{L^2_{\C}(\Omega_d)}\neq 0$.
Then
\[
\dim\ker\J_p(\psi_{\text{s}};\mu_{\text{s}}) < \dim\ker\J(\psi_{\text{s}};\mu_{\text{s}}).
\]
\end{corollary}
\begin{proof}
  The corollary is a direct consequence of lemma~\ref{lemma:keller} (page~\pageref{lemma:keller})
  so it suffices here to verify that it can be applied on $\J_p(\psi_;\mu)$.
First note that the phase condition
$\Im(\left\langle\psi_0,\varphi\right\rangle_{L^2_{\C}(\Omega_d)})$
is a linear functional over the $\R$-vector space $X_d$. Furthermore, we
have by~(\ref{eq:jacobian nullspace}) that $\spn\{\i\psi_{\text{s}}\}\subseteq \ker\J(\psi_{\text{s}};\mu_{\text{s}})$.
Evaluating the phase condition at $\i\psi_{\text{s}}$ yields
\[
\Im\left(\left\langle\psi_0,\i\psi_{\text{s}}\right\rangle_{L^2_{\C}(\Omega_d)}\right)
= \left\langle\psi_0,\psi_{\text{s}}\right\rangle_{L^2_{\C}(\Omega_d)}
\neq 0
\]
by assumption. Moreover, corollary~\ref{corollary:j self-adjoint} states
that $\J(\psi_{\text{s}};\mu_{\text{s}})=\J^*(\psi_{\text{s}};\mu_{\text{s}})$ with respect to the inner product~(\ref{eq:real inner product}).
From this, it follows that
\[
\range(\J(\psi_{\text{s}};\mu_{\text{s}})) = \ker(\J^*(\psi_{\text{s}};\mu_{\text{s}}))^\bot = \ker(\J(\psi_{\text{s}};\mu_{\text{s}}))^\bot,
\]
so to show that $b=-\i\psi_{\text{s}}\notin\range(\J)$ is suffices to show that
$b$ is not orthogonal to all of $\ker(\J(\psi_{\text{s}};\mu_{\text{s}}))$
with respect to the inner product~(\ref{eq:real inner product}). This holds true 
for $\psi_{\text{s}}\neq 0$ since
\[
\left\langle \i\psi_{\text{s}}, \i\psi_{\text{s}} \right\rangle_{\R}
= \Re \left\langle \psi_{\text{s}},\psi_{\text{s}} \right\rangle_{L^2_{\C}(\Omega_d)}
= \left\|\psi_{\text{s}}\right\|_{L^2_{\C}(\Omega_d)}^2 \neq 0.
\]
Thus, all conditions of lemma~\ref{lemma:keller} are fulfilled and its application to
$\J_p(\psi;\mu)$ concludes the proof.
\end{proof}

In the remainder of this section, we show that the extended operator retains the
symmetries of the Ginzburg--Landau system.

\paragraph{Symmetries of the extended system}\label{sec:symmetries extended sys}
Given $\gamma\in\Gamma$, the symmetries $\tilde{\gamma}$ of the extended system are defined to act
\begin{equation}\label{eq:extended symmetries}
\tilde{\gamma}
\begin{pmatrix}
\psi\\
\eta
\end{pmatrix}
\dfn
\begin{pmatrix}
\gamma \psi\\
\eta
\end{pmatrix}.
\end{equation}
\begin{lemma}\label{lemma:extended symmetries}
The extended system (\ref{eq:extended GL}) is equivariant exactly
under all actions in $\Gamma$ that leave $\psi_0$ invariant, i.e.,
$\tilde{\Gamma}=\Sigma_{\psi_0}\cap\Gamma$.
\end{lemma}
\begin{proof}
We have to proof equivariance only for the generators $\tilde{\rho}, \tilde{\sigma}\in\widetilde{\Gamma}$.
For a given $\gamma\in\Sigma_{\psi_0}\cap\Gamma$, it has to be shown that
\[
\begin{pmatrix}
\gamma \left[\GL(\psi ) - \i\eta \psi \right]\\
\Im(\left\langle\psi_0,\psi\right\rangle_{L^2_{\C}(\Omega_d)})
\end{pmatrix}
=
\tilde{\gamma}\GL_p(\psi,\eta)
\stackrel{!}{=}
\GL_p( \tilde{\gamma}(\psi,\eta))
=
\begin{pmatrix}
 \GL(\gamma \psi ) - \i\eta (\gamma\psi) \\
\Im(\left\langle\psi_0,\gamma\psi\right\rangle_{L^2_{\C}(\Omega_d)})
\end{pmatrix},
\]
which holds obviously true for the first component, owing to the $\Gamma$-invariance
of $\GL$.
As for the second component, we have
\[
\Im(\left\langle\psi_0,\psi\right\rangle_{L^2_{\C}(\Omega_d)})
\stackrel{!}{=}
\Im(\left\langle\psi_0,\gamma\psi\right\rangle_{L^2_{\C}(\Omega_d)})
=
\Im\left( \int_\Omega \conj{\psi}_0 (\gamma\psi) \,\mathrm{d}\Omega \right).
\]
After a suitable change of variables, this is equivalent to show that 
\begin{equation}\label{eq:phase condition 2}
\Im(\left\langle\psi_0,\psi\right\rangle_{L^2_{\C}(\Omega_d)})
\stackrel{!}{=}
\Im\left( \int_\Omega (\gamma^{-1}\conj{\psi}_0) \psi \det \vartheta_{\gamma}
\,\mathrm{d}\Omega \right) \quad \forall\psi\in L^2_{\C}(\Omega)
\end{equation}
holds exactly for all $\gamma\in\Sigma_{\psi_0}\cap\Gamma$.

Firstly, let us show this equivalence for the cyclic subgroup $C_4 \preceq \Sigma_{\psi_0}$.
Given $\gamma\in C_4$ (and thus $\gamma\psi_0 = \psi_0$, $\det\vartheta_{\gamma}=1$),
equation~(\ref{eq:phase condition 2}) obviously holds true.
On the other hand, let us assume that equation~(\ref{eq:phase condition 2}) is valid
and let us take a sequence of Dirac-$\delta$ functions centered at $(x_0, y_0)$,
$\psi^{(l)}=\delta_{(x_0,y_0)}^{l}$.
Then
\begin{multline*}
\Im\left(\conj{\psi}_0(x_0,y_0)\right)
=
\lim\limits_{l\to\infty}\Im\left(\int_\Omega \conj{\psi}_0 \psi^{(l)}\,\mathrm{d}\Omega \right)
\stackrel{!}{=}\\
\lim\limits_{l\to\infty}\Im\left(\int_\Omega \gamma^{-1}\conj{\psi}_0 \psi^{(l)} \,\mathrm{d}\Omega\right)
=
\Im\left( (\gamma^{-1}\conj{\psi}_0\right)(x_0,y_0))
\end{multline*}
Since this can be done for any $(x_0,y_0)\in\Omega\cup\partial\Omega$,
we have $\Im(\psi_0)\in C_4$. The same result is obtained for
$\Re(\psi_0)$ by taking $\psi^{(l)}=\i \delta_{(x_0,y_0)}^{l}$.
We conclude that $\Sigma_{\psi_0}\supseteq C_4$ is also necessary for (\ref{eq:phase condition 2}) to hold.

The very same arguments can be applied to the conjugate reflection $\sigma$, noting
that $\det \vartheta_{\sigma}=-1$, and that the action of $\sigma$ changes the sign
of the expression.
\end{proof}

The choice of $\psi_0$ must hence be such that it eliminates the phase
invariance (according to corollary~\ref{cor:nullspace reduction}),
and that it preserves the other symmetries of the system (according to
lemma~\ref{lemma:extended symmetries}).
The first condition is equivalent to demanding
$0\neq\left\langle\psi_0,\psi_{\text{s}}\right\rangle$ which indeed is a rather mild
condition that will be fulfilled, for instance, by $\psi_0\equiv 1$ for most
scenarios considered later.
Note that it is also possible to update $\psi_0$ in each Newton step to the current
guess $\psi^{(k)}$: For a solution $\psi_\text{s}$, let $\psi_\text{s}=\psi^{(k)}+e^{(k)}$; we have 
\[
\left\langle \psi^{(k)},\psi_\text{s} \right\rangle = \|\psi_\text{s}\|^2_{L^2_{\C}(\Omega_d)} + \left\langle e^{(k)}, \psi  \right\rangle,
\]
which is guaranteed to be  nonzero for sufficiently small $e^{(k)}$ if $\psi_\text{s}\neq0$.
Note that intermediate Newton steps might not exactly preserve the symmetries
of the system, but the symmetry breaking is weak in the sense that symmetry
is preserved at convergence, and will not do harm \cite{Hoyle:2006:PF}.

\section{The discretized system}\label{sec:discretized}

An important property of the full Ginzburg--Landau equations is its \emph{gauge
invariance}, a generalization of the phase symmetry (\ref{eq:reduced gauge invariance}) to
the case where $\A$ is not fixed (see section 3.1 in \cite{DGP:1992:AAG}).
While the reduced invariance with fixed
$\A$ is preserved under all consistent
pointwise discretizations of the Ginzburg--Landau equations
\cite{Du:1998:DGI,DJ:2004:NSQ,KK:1995:VCT},
ordinary finite-difference discretizations lead to systems that are gauge invariant
only up to $O(h)$, where $h$ is the grid spacing.
It is thus customary to reformulate the equations using techniques from lattice gauge theory.
For the convenience of the reader, the new system will be presented in this
section. We also show that, for appropriate phase conditions,
all the symmetries are preserved in the extended discretized system.

\subsection{Formulation with link variables}
Let us consider the functions
\begin{align*}
U_x(x,y) &\dfn \exp\left( -\i \igralnl{x_0}{x}{A_x(\xi,y)}{\xi}  \right),\\
U_y(x,y) &\dfn \exp\left( -\i \igralnl{y_0}{y}{A_y(x,\nu)}{\nu}  \right),
\end{align*}
with arbitrary, fixed $x_0, y_0\in\R$.
It can be checked easily that
\[
\sum\limits_{\nu\in\{x,y\}}
- \conj{U_{\nu}(x,y)} \frac{\partial^2}{\partial {\nu}^2}(U_{\nu}\psi)
=
( -\i \bn - \A)^2 \psi - \i (\bn\cdot\A) \psi.
\]
The Ginzburg--Landau equations (\ref{eq:GL}) can then be written as
\begin{equation}\label{eq:GL link variables}
\begin{cases}
0 = - \displaystyle\sum\limits_{\nu\in\{x,y\}}
\conj{U_{\nu}\psi} \,\textstyle\frac{\partial^2}{\partial {\nu}^2}(U_{\nu}\psi) - \conj{\psi}\psi (1 - \conj{\psi} \psi) \quad \text{in } \Omega,\\[4.0ex]
0 = \n \cdot
\begin{pmatrix}
\conj{U_x\psi}\, \frac{\partial}{\partial x}(U_x\psi)\\[1ex]
\conj{U_y\psi}\, \frac{\partial}{\partial y}(U_y\psi)
\end{pmatrix} \quad \text{on } \partial\Omega.
\end{cases}
\end{equation}
$U_{\nu}$ only appears in the product $U_{\nu}\psi$ which guarantees
preservation of full gauge invariance \cite{DGP:1992:AAG}.

\subsubsection{Discretization}

Let $N>0$, for simplicity even, $h = d/N$, and let $\Omega_h=\{ h \cdot (i,j) \,|\,
-N/2\le i\le N/2, -N/2\le j\le N/2 \}$ be a uniform grid on $\Omega_d$.
Furthermore, let us consider the discretization $\psi^{(h)} \in X_h=\C^{(N+1)\times
(N+1)}$. The ordinary five-point discretization of (\ref{eq:GL link variables}) with
centered finite differences on the boundaries is given by
\begin{multline}\label{eq:GL discretized}
0
= \GL^{(h)}(\psi^{(h)};\mu)
\dfn \left(D_{xx} \psi^{(h)}\right)_{i,j} + \left(D_{yy} \psi^{(h)}\right)_{i,j} - \psi^{(h)}_{i,j} \left( 1- \conj{\psi^{(h)}_{i,j}} \psi^{(h)}_{i,j} \right)\\
\forall i,j\in\{-N/2,\dots,N/2\}
\end{multline}
where the finite-difference operator $D_{xx}$ is defined by
\begin{multline*}
\left(D_{xx} \psi^{(h)}\right)_{i,j}
\dfn\\
\begin{cases}
h^{-2}\left( \phantom{{}-2(U_{x}^{(h)})_{i+1,j} \psi^{(h)}_{i+1,j} +{}} 2 \psi^{(h)}_{i,j} - 2(U_{x}^{(h)})_{i-1,j}  \psi^{(h)}_{i-1,j}\right)&\text{for } i=N/2,\\[1ex]
h^{-2}\left( \phantom{2}-(U_{x}^{(h)})_{i+1,j} \psi^{(h)}_{i+1,j} + 2 \psi^{(h)}_{i,j} - (U_{x}^{(h)})_{i-1,j}  \psi^{(h)}_{i-1,j}\right)&\text{for } -\frac{N}{2}<i<\frac{N}{2},\\[1ex]
h^{-2}\left( {}-2(U_{x}^{(h)})_{i+1,j} \psi^{(h)}_{i+1,j} + 2 \psi^{(h)}_{i,j} \right)&\text{for } i=-N/2,
\end{cases}
\end{multline*}
and likewise for $D_{yy} \psi^{(h)}$, with unknown $\psi^{(h)}\in X_h$, where
\[
(U_x^{(h)})_{i\pm 1,j} = \exp\left( -\i I_{x_i}^{x_{i\pm 1}}(A_x(\cdot,y_j)) \right)  + O(h^p)
\]
with $I_{x_i}^{x_{i\pm1}}(A_x(\cdot,y_j))\approx\igralnl{x_i}{x_{i\pm1}}{A_x(\xi,y_j)}{\xi}$,
and likewise for $(U_y^{(h)})_{x,j\pm 1}$. The order $p$
of the approximation depends on the quadrature method in use.
It is easy to verify that the discretization (\ref{eq:GL discretized}) has order of consistency $\min\{2,p-2\}$.

The discrete Jacobian operator $\J^{(h)}$ is defined similarly as
\begin{equation}\label{eq:discrete jacobian}
\begin{split}
&\J^{(h)}(\psi^{(h)};\mu): X_h \to X_h,\\
&\left(\J^{(h)}(\psi^{(h)};\mu)\,\varphi^{(h)}\right)_{i,j}
\dfn \left(D_{xx} \psi^{(h)}\right)_{i,j} + \left(D_{yy} \psi^{(h)}\right)_{i,j} - \varphi^{(h)}_{i,j} + 2\left|\psi^{(h)}_{i,j}\right|^2 \varphi^{(h)}_{i,j}\\
&\phantom{\left(\J^{(h)}(\psi^{(h)};\mu)\varphi^{(h)}\right)_{i,j}}\relphantom{\dfn}
+  \left(\psi^{(h)}_{i,j}\right)^2 \conj{\varphi^{(h)}_{i,j}}
\end{split}
\end{equation}
and it is self-adjoint
with respect to the scalar product $\Re\langle\cdot,\cdot\rangle_{\C}$. This is
a consequence of lemma~\ref{lemma: hilbert space adjointness} upon realizing
that the operators $D_{xx}$ and $D_{yy}$ are both self-adjoint with
respect to the inner product $\langle\cdot,\cdot\rangle_{\C}$ in $\C^{(N+1)\times(N+1)}$.
A consequence of this is that all eigenvalues of the operator~(\ref{eq:discrete jacobian})
are real-valued.

\subsubsection{Symmetries of the discretized system}
Now that the structure of the discretized operator is described, we review
the symmetries of the associated boundary-value problem.
Many of the results in this section can be borrowed from
Section~\ref{subsection:regularization} on the continuous problem with little
modification. For example, the discretized system (\ref{eq:GL discretized})
is invariant under
\[
\left(\theta_{\eta}^{(h)}\psi^{(h)}\right)_{i,j} \dfn \exp(\i\eta)\psi^{(h)}_{i,j}
\]
pointwise for each $\eta\in[0,2\pi)$.
Let further the discrete symmetry operators $\rho$ and $\sigma$ be defined by
\begin{gather}\label{eq:discretized symmetry actions}
\left(\rho^{(h)}\psi^{(h)}\right)_{i,j} \dfn \psi^{(h)}_{j,-i},\qquad
\left(\sigma^{(h)}\psi^{(h)}\right)_{i,j} \dfn \conj{\psi^{(h)}_{-i,j}}.
\end{gather}
It can easily be shown that the discretized system (\ref{eq:GL discretized}) is invariant under these actions.

Just like in the continuous case (\ref{eq:jacobian nullspace}), the discrete phase invariance induces a nontrivial nullspace of $\J^{(h)}$:
\begin{equation}\label{eq:discretized nullspace}
\begin{split}
\left(\J^{(h)}(\psi^{(h)};\mu)\,\i\psi^{(h)}\right)_{i,j}
&= \i\left(D_{xx} \psi^{(h)}\right)_{i,j} + \i\left(D_{yy} \psi^{(h)}\right)_{i,j} - \i\psi^{(h)}_{i,j} + 2\i\left|\psi^{(h)}_{i,j}\right|^2 \psi^{(h)}_{i,j}\\
&\relphantom{=} - \i \left(\psi^{(h)}_{i,j}\right)^2 \conj{\psi^{(h)}_{i,j}}\\
&= \i \psi^{(h)}_{i,j} \left( 1- \conj{\psi^{(h)}_{i,j}} \psi^{(h)}_{i,j} \right)  - \i\psi^{(h)}_{i,j} + \i\left|\psi^{(h)}_{i,j}\right|^2 \psi^{(h)}_{i,j}\\
&= 0.
\end{split}
\end{equation}
and hence
$\i\psi^{(h)}_{i,j}\in\ker\J^{(h)}(\psi_{\text{s}}^{(h)}\rev{;}\mu_{\text{s}})$.
This makes it impossible to treat the system~\eqref{eq:GL discretized}
with generic linear solvers \emph{at} a solution
$\psi_{\text{s}}^{(h)}$, as the system is then exactly singular.
In the neighborhood of $\psi_{\text{s}}^{(h)}$, the
system will have a large condition number, making round-off errors
dominate the update term~\cite{demmel} which in turn flaws the next Newton
step. This phenomenon is illustrated in Figure~\ref{fig:Newton without
bordering}.

As suggested in section~\ref{subsection:regularization},
we will avoid the singularity of the Jacobian by using a phase condition,
which in its discretized form reads
\begin{equation}\label{eq:discretePhaseCondition}
I^{(h)}(\psi_0^{(h)},\psi^{(h)})
\dfn
\Im\left( \sum\limits_{i=-N/2}^{N/2}w_i h \sum\limits_{j=-N/2}^{N/2} w_j h  \: \conj{\left(\psi^{(h)}_{0}\right)_{i,j}} \psi^{(h)}_{i,j}\right),
\end{equation}
where $\psi_0^{(h)}\in X_h$ is a given reference state, and
\[
w_i =
\begin{cases}
1 \quad& \text{for } -N/2<i<N/2,\\
\frac{1}{2}\quad& \text{for } i\in\{-N/2,N/2\}.
\end{cases}
\]
The discretized version of the extended system is then
\begin{equation}\label{eq:discretized extended GL}
0 =
\GL^{(h)}_p(\psi^{(h)},\eta;\mu)
\dfn
\begin{pmatrix}
\GL^{(h)}(\psi^{(h)};\mu) - \i\eta\psi^{(h)}\\
I^{(h)}(\psi_0^{(h)},\psi^{(h)})
\end{pmatrix},
\end{equation}
and the symmetry operations for this extended system can be defined just
like in~(\ref{eq:extended symmetries}).

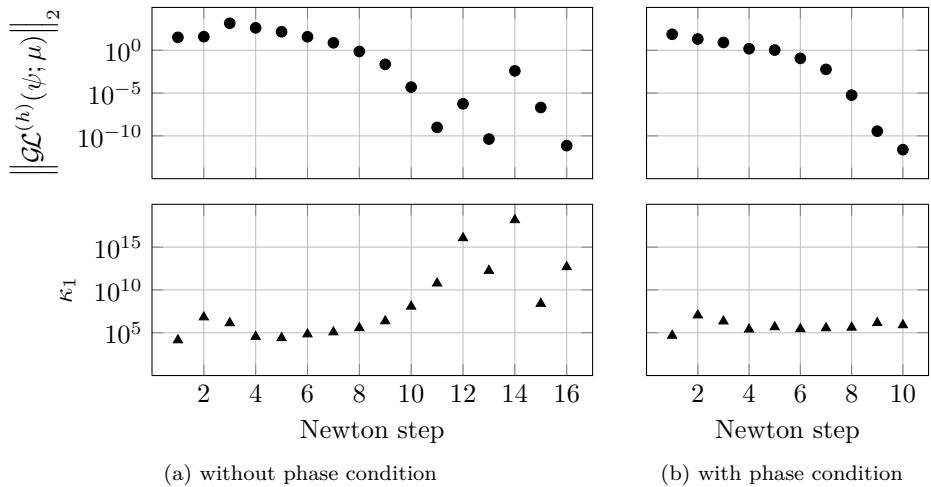
\begin{figure}
\centering

\subfloat[][without phase condition]{\begin{tikzpicture}

\begin{semilogyaxis}
[axis on top,
xtick={2,4,6,8,10,12,14,16},
xticklabels=\empty,
ytick={1e-10,1e-05,1},
xmin=0.000000e+00,xmax=1.700000e+01,
ymin=1.000000e-15,ymax=1.000000e+05,
xmajorgrids,
ymajorgrids,
ylabel={$\norm{\GL^{(h)}(\psi;\mu)}_2$},
width=\figurewidth,
height=0.7\figureheight,
scale only axis,
name=Fx
]

\addplot [only marks,mark=*,mark options={solid}] coordinates{
 (1.000000e+00,3.206000e+01) (2.000000e+00,3.860000e+01) (3.000000e+00,1.421000e+03) (4.000000e+00,4.143000e+02) (5.000000e+00,1.445000e+02) (6.000000e+00,3.775000e+01) (7.000000e+00,7.455000e+00) (8.000000e+00,7.228000e-01) (9.000000e+00,2.275000e-02) (1.000000e+01,4.953000e-05) (1.100000e+01,9.718000e-10) (1.200000e+01,5.534000e-07) (1.300000e+01,4.217000e-11) (1.400000e+01,3.930000e-03) (1.500000e+01,2.067000e-07) (1.600000e+01,7.231000e-12)
};

\end{semilogyaxis}

\path (Fx.below south west) ++(0,-0.1cm) coordinate (lower plot position);

\begin{semilogyaxis}[
at={(lower plot position)},
anchor=north west,
xtick={2,4,6,8,10,12,14,16},
ytick={1e+05,1e+10,1e+15},
xmin=0.000000e+00,xmax=1.700000e+01,
ymin=1.000000e+00,ymax=1.000000e+20,
xmajorgrids,ymajorgrids,
ylabel={$\kappa_1$},
width=\figurewidth,
height=0.7\figureheight,
scale only axis,
xlabel={\nxs{Newton step}}
]

\addplot [only marks,mark=triangle*,mark options={black,solid}] coordinates{
 (1.000000e+00,1.373000e+04) (2.000000e+00,6.555000e+06) (3.000000e+00,1.365000e+06) (4.000000e+00,3.484000e+04) (5.000000e+00,2.476000e+04) (6.000000e+00,6.740000e+04) (7.000000e+00,1.172000e+05) (8.000000e+00,3.646000e+05) (9.000000e+00,2.289000e+06) (1.000000e+01,1.174000e+08) (1.100000e+01,5.928000e+10) (1.200000e+01,1.163000e+16) (1.300000e+01,1.781000e+12) (1.400000e+01,1.481000e+18) (1.500000e+01,2.501000e+08) (1.600000e+01,4.803000e+12)
};

\end{semilogyaxis}

\end{tikzpicture}\label{fig:Newton without bordering}}
\quad
\subfloat[][with phase condition]{\begin{tikzpicture}

\begin{semilogyaxis}
[axis on top,
xtick={2,4,6,8,10,12,14,16},
xticklabels=\empty,
ytick={1e-10,1e-05,1},
yticklabels=\empty,
xmin=0.000000e+00,xmax=1.100000e+01,
ymin=1.000000e-15,ymax=1.000000e+05,
xmajorgrids,
ymajorgrids,
width=0.64\figurewidth,
height=0.7\figureheight,
scale only axis,
name=Fx
]

\addplot [only marks,mark=*,mark options={solid}] coordinates{
  (1.000000e+00,7.335000e+01) (2.000000e+00,2.060000e+01)
 (3.000000e+00,8.150000e+00) (4.000000e+00,1.499000e+00)
 (5.000000e+00,1.057000e+00) (6.000000e+00,1.125000e-01)
 (7.000000e+00,5.806000e-03) (8.000000e+00,5.687000e-06)
 (9.000000e+00,3.536000e-10) (1.000000e+01,2.450000e-12)
};

\end{semilogyaxis}

\path (Fx.below south west) ++(0,-0.1cm) coordinate (lower plot position);

\begin{semilogyaxis}[
at={(lower plot position)},
anchor=north west,
xtick={2,4,6,8,10,12,14,16},
ytick={1e+05,1e+10,1e+15},
yticklabels=\empty,
xmin=0.000000e+00,xmax=1.100000e+01,
ymin=1.000000e+00,ymax=1.000000e+20,
xmajorgrids,ymajorgrids,
width=0.64\figurewidth,
height=0.7\figureheight,
scale only axis,
xlabel={\nxs{Newton step}}
]

\addplot [only marks,mark=triangle*,mark options={solid,black}] coordinates{
 (1.000000e+00,4.677000e+04) (2.000000e+00,1.112000e+07) (3.000000e+00,2.185000e+06) (4.000000e+00,2.494000e+05) (5.000000e+00,4.575000e+05) (6.000000e+00,2.756000e+05) (7.000000e+00,3.510000e+05) (8.000000e+00,4.113000e+05) (9.000000e+00,1.422000e+06) (1.000000e+01,7.904000e+05)
};

\end{semilogyaxis}

\end{tikzpicture}\label{fig:Newton with bordering}}

\caption{The residual of the Newton iterations and the condition number
\nxs{$\kappa_1$} of the \nxs{associated} linear system in the $1$-norm.
(a) Here, the linear Jacobian systems are solved using Gaussian elimination and deliver
flawed Newton updates as the condition number increases. This is a principle
problem and it is not limited to Gaussian elimination.
(b) The regularization removes the singularity and this leads to bounded condition
numbers. The linear systems can then be solved accurately.}
\label{fig:Newton iterations}
\end{figure}

Parallel to~(\ref{eq:extended jacobian}), the discrete extended Jacobian operator $\J^{(h)}$ is
\begin{equation}\label{eq:discrete extended jacobian}
\begin{split}
&\J^{(h)}_p(\psi^{(h)};\eta;\mu): X_h\times\R \to X_h\times\R,\\
&\left(\J_p^{(h)}(\psi^{(h)};\eta;\mu)
\begin{pmatrix}
\varphi^{(h)}\\
\nu
\end{pmatrix}
\right)_{i,j}
\dfn
\begin{pmatrix}
(\J^{(h)}(\psi^{(h)};\mu)-\i\eta)\varphi^{(h)} - \i\psi^{(h)}\nu\\
I^{(h)}(\psi_0^{(h)},\varphi^{(h)})
\end{pmatrix}
\end{split}
\end{equation}

The following two statements are the discrete versions of the
lemmas~\ref{cor:nullspace reduction} and \ref{lemma:extended symmetries}.

\begin{corollary}
Let $(\psi_{\mathrm{s}}^{(h)},\mu_{\mathrm{s}})\in X_h\times\R$ be a solution of the
original discretized Ginzburg--Landau equations~(\ref{eq:GL discretized}) with $\psi_{\mathrm{s}}^{(h)}\neq 0$, and let $\psi_0^{(h)}\in X_h$
such that $\left\langle\psi_{0}^{(h)},\psi_{\mathrm{s}}^{(h)}\right\rangle_{\C}\neq 0$.
Then
\[
\dim\ker\J_p^{(h)}(\psi_{\mathrm{s}}^{(h)};\mu_{\mathrm{s}}) < \dim\ker\J^{(h)}(\psi_{\mathrm{s}}^{(h)};\mu_{\mathrm{s}}).
\]
\end{corollary}
\begin{proof}
The proof runs parallel to the one of corollary~\ref{cor:nullspace reduction},
using the discrete inner product $\Re\langle\cdot,\cdot\rangle_{\C}$.
\end{proof}

\begin{lemma}
The extended equations (\ref{eq:discretized extended GL}) are equivariant exactly under $\Sigma_{\psi_0^{(h)}}\cap\Gamma^{(h)}$.
\end{lemma}
\begin{proof}
Again, the proof is essentially parallel to the one of lemma~\ref{lemma:extended symmetries};
instead of series of Dirac-$\delta$ function $\delta_{(x_0,y_0)}$, we can use their discrete
equivalents
\[
\varphi_{i,j}^{(h)}
\dfn
\begin{Bmatrix}
1\:&\text{for } i=i_0, j=j_0\\
0\:&\text{otherwise}
\end{Bmatrix},
\quad
\widetilde{\varphi}_{i,j}^{(h)}
\dfn
\i\,\varphi_{i,j}^{(h)}.
\]
\end{proof}

\section{Numerical results}\label{sec:results}

Using the framework presented in the previous sections, it is possible to solve the
Ginzburg--Landau equations numerically for any given parameter $\mu$ (the strength
of the applied magnetic field) and $d$ (the edge length of the sample). 
As discussed in Section~\ref{sec:introduction}, the intensity of the applied magnetic
field, 
can be tuned experimentally, and it is thus interesting to explore the bifurcation
scenario as this parameter is varied. 

Because of the symmetries of the Ginzburg--Landau system posed on the square,
we expect symmetry-breaking bifurcations to arise. As described in Section
\ref{sec:symmetries extended sys}, the extended system \eqref{eq:extended GL} does not bear
the continuous $\T$-symmetry such that the relevant symmetry group for our
computations is \Dfour.

Symmetry-breaking bifurcations in \Dfour{} are well known \cite{Hoyle:2006:PF,golubitsky2002symmetry}.
We recall here that, in our case, the group generators are $\rho$, the rotation by
$\pi/2$ (see equation~\eqref{eq:rotation}), and $\sigma$, the conjugated mirroring
along the $y$-axis (see equation~\eqref{eq:mirror}). We expect that symmetry-breaking
bifurcations will occur when critical eigenvalues become unstable with (algebraic and
geometric) multiplicity either 1 or 2. With a simple unstable eigenvalue, one should
expect either a symmetry-preserving turning point or a pitchfork bifurcation with
branches corresponding to the four one-dimensional irreducible representations of
\Dfour. With an eigenvalue of multiplicity 2 crossing the origin, two families of
branches emerge from the bifurcation point, corresponding to the conjugacy classes of
the \Dfour{} isotropy subgroups $\langle \rho\rangle$ and $\langle
\sigma\rho\rangle$, respectively.

\begin{figure} 
  \centering
  \includegraphics[width=0.8\textwidth]{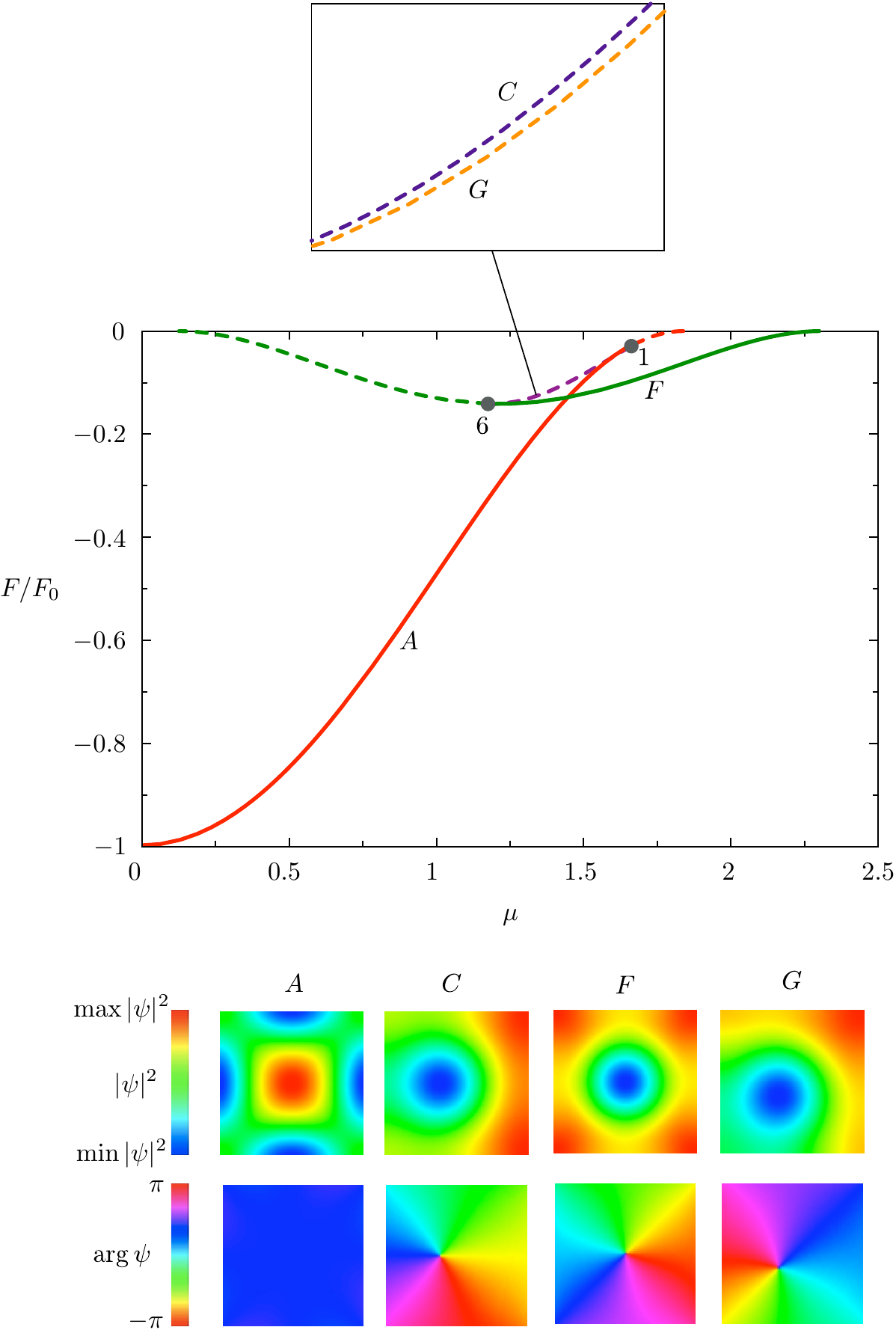}
  \caption{Free energy of the solutions as a function of the strength of the applied
  magnetic field $\mu$. Solid (dashed) lines represent stable (unstable) states.
  Branch $A$ with \Dfour{} symmetry starts at the homogeneous
  state with zero field and becomes unstable at bifurcation point 1. Branches $G$ and
  $C$ emerge from the bifurcation point, both characterized by a single vortex
  entering the domain; branch $C$ has mirror symmetry along either the horizontal or
  the vertical center line, while branch $G$ has mirror symmetry along one of the
  diagonals of the square domain. At bifurcation point $6$, branches $G$ and $C$
  connect to \nxs{$F$}, characterized by solutions with a single vortex in the center of
  the domain and full \Dfour{}-symmetry.} 
  \label{fig:solutionsmallbifurcationdiagram}
\end{figure}

\begin{figure}
  \centering
  \includegraphics[width=1.0\textwidth]{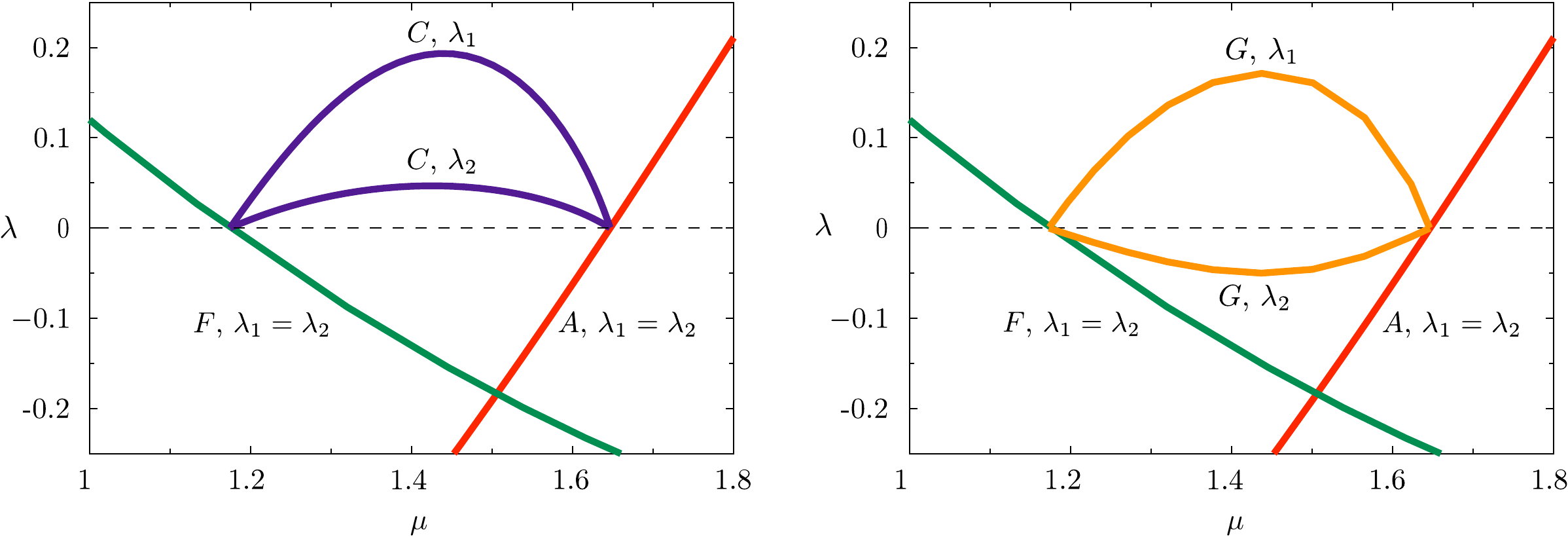}
  \caption{The two largest eigenvalues of the Jacobian as a function of
  the applied magnetic field. We plot the eigenvalues close to bifurcation points $1$
  and $6$, for each of the four solution branches in
  Figure~\ref{fig:solutionsmallbifurcationdiagram}. For field strength above
  $\mu=1.646$ the main branch is unstable, while for fields strengths weaker than
  $\mu=1.175$ the branch with a single vortex is unstable.  We see that the largest eigenvalue of the main
  branch, which has multiplicity two, splits into two separate eigenvalues.  Curve
  $C$ has two unstable eigenmodes, while $B$ has one stable and one unstable eigenmode.
  The colors reflect the branches in Figure~\ref{fig:solutionsmallbifurcationdiagram}.}
  \label{fig:solutionsmallbifurcationdiagrameigenvalues}
\end{figure}

In the present section, the parameter $\mu$ will be varied for two different domain
sizes $d$. The simplest nontrivial example of symmetry-breaking bifurcation occurs for
small domain sizes, so we have deliberately chosen $d=3.0$,
a domain size that is just enough to host a single vortex.
Subsequently, we study the case $d=5.5$, for which
the bifurcation scenario becomes increasingly more involved.

The bifurcation diagrams are traced via standard numerical continuation methods
\cite{krauskopf2007}. The technical implementation is based on the
Trilinos project~\cite{HW:2003:TUG}
and exploits the sparse structure of the
discrete Jacobian operator \eqref{eq:discrete jacobian} as well as its properties,
as outlined in Section~\ref{sec:self adjointness}.

In the remainder of this section, we will denote solution branches (and relative
patterns) alphabetically and bifurcation points with numerals.

\begin{remark}\label{remark:energy}
Unless otherwise stated, the bifurcation diagrams are plotted in terms of
the expression
\begin{equation}\label{eq:gibbs portion}
F(\psi,\mu) \dfn F_{\text{max}}^{-1}\, \xi\frac{|\alpha|^2}{\beta}
\int_{\Omega}
-|\psi|^2
+ \frac{1}{2}|\psi|^4
+ \left|-\i\bn\psi - \A(\mu)\psi \right|^2\,\mathrm{d}\Omega
\end{equation}
which is part of the Gibbs energy~(\ref{eq:Gibbs}).
This is in accordance to what is usually done in the physics literature.
Applying
(\ref{eq:keo property}) in the case $\psi\in X_d$, we obtain
\[
F(\psi,\mu) = F_{\text{max}}^{-1}\, \xi\frac{|\alpha|^2}{\beta}
\left[\int_{\Omega}
-|\psi|^2
+ \frac{1}{2}|\psi|^4
+ \int_{\Omega} \overline{\psi} (-\i\bn\psi - \A(\mu))^2\psi\,\mathrm{d}\Omega
\right].
\]
Only solutions $\psi(\mu)$
of the Ginzburg--Landau equations~(\ref{eq:GL}) are considered, so that 
\begin{multline*}
F(\psi,\mu) = F_{\text{max}}^{-1}\, \xi\frac{|\alpha|^2}{\beta}
\left[
\int_{\Omega}
-|\psi|^2
+ \frac{1}{2}|\psi|^4
+ \int_{\Omega} \overline{\psi} \psi(1-|\psi|^2)\,\mathrm{d}\Omega
\right]\\
=
- F_\text{max}^{-1}\, \xi\frac{|\alpha|^2}{2\beta}
\int_{\Omega}
|\psi|^4\,\mathrm{d}\Omega.
\end{multline*}
Thus, computing the significant portion~(\ref{eq:gibbs portion})
of the Gibbs energy~(\ref{eq:Gibbs}) effectively reduces to
evaluating
\[
F(\psi,\mu) = - |\Omega|^{-1} \int_{\Omega} |\psi|^4\,\mathrm{d}\Omega.
\]
\end{remark}

\subsection{Small-sized system ($d=3$)}\label{sec:smallsystem}
The first computed solution corresponds to
a superconductor in the absence of a magnetic field, that is,
$\mu=0$;
the system is in the homogeneous solution $\psi\equiv 1$ and
it is said to be in a completely superconducting state. The solution has
all the symmetries of the system (its isotropy subgroup is the full group
\Dfour) and is stable as a global minimum of the free energy
(\ref{eq:Gibbs}).

With the help of numerical continuation, a series of solutions for increasing $\mu$
is constructed. This results in branch $A$ in
Figures~\ref{fig:solutionsmallbifurcationdiagram} and
\ref{fig:solutionsmallbifurcationdiagrameigenvalues}, showing the energy of the
solution and the two most unstable eigenvalues of the Jacobian as a function of the
field strength, respectively. For non-zero field strength, the solutions deviate from
the homogeneous superconducting state, 
developing zones of low supercurrent density near the edges of the domain (see
pattern $A$ in Figure~\ref{fig:solutionsmallbifurcationdiagram}). As $\mu$ is increased,
the states are characterized by a higher energy and they maintain full
\Dfour{} symmetry.

At field strength $\mu\approx1.64$ (point $1$ in
Figure~\ref{fig:solutionsmallbifurcationdiagram}), an eigenvalue with
multiplicity 2 becomes unstable. At this bifurcation point, one can apply the
equivariant branching lemma: the Ginzburg--Landau equation is equivariant under the
symmetries of the finite group $D_4$ and the eigenvalues cross the origin with
non-zero speed, (see Figure~\ref{fig:solutionsmallbifurcationdiagrameigenvalues}).  
The lemma guarantees the existence of two solution branches
emerging from the bifurcation, corresponding to the conjugacy classes of the isotropy
subgroups $\langle \sigma \rangle$ and $\langle \sigma \rho \rangle$. They both have
a one-dimensional fixed-point subspace.
Hence, we expect two \emph{different} families of solution branches, each containing
four equivalent bifurcation curves with states belonging to one group orbit. The two
families are found in the branches $G$ and $C$ of
Figure~\ref{fig:solutionsmallbifurcationdiagram}.

Before describing curves $G$ and $C$, the two curves that emerge from the
bifurcation point, we continue
to follow the original branch $A$ for increasing $\mu$. The state is now unstable
and retains full $D_4$ symmetry. The magnetic field
penetration increases from the boundaries until, at $\mu\approx 1.89$, the
branch connects to the trivial state $\psi\equiv 0$, which corresponds to the normal
state of the sample.

We now discuss the curves $G$ and $C$, which have reduced symmetry and emerge from
the bifurcation points $1$ and $6$. Curve $C$ corresponds to solutions in which a
single vortex moves in from one of the four sides of the square.  These solutions
belong to the conjugacy class of the subgroup $\langle\sigma\rangle$ and the single
vortex sits either on the horizontally or vertically centered line.

The other family of solutions, on branch $G$, also features a single vortex entering the
system, but along one of the diagonals. These solutions have an isotropy subgroup
that belongs to the conjugacy class of $\langle \sigma\rho \rangle$, hence their
symmetry with respect to one of the diagonals.

Solutions belonging to curves $G$ and $C$ are energetically similar, the latter
having slightly higher energy, as it can be seen from the inset of
Figure~\ref{fig:solutionsmallbifurcationdiagram}.

As we decrease the field strength from point $1$ to point $6$, the vortex moves along
the center line for curve $C$, or along the diagonal for curve $G$, towards the center of
the sample. At field strength $\mu\approx 1.18$ (point $6$ in
Figure~\ref{fig:solutionsmallbifurcationdiagram}), each solution features a vortex in
the middle of the domain and enjoys full \Dfour{}-symmetry. As we can see in
Figures~\ref{fig:solutionsmallbifurcationdiagram} and
\ref{fig:solutionsmallbifurcationdiagrameigenvalues}, bifurcation point $6$ is
analogous to bifurcation point $1$, but it involves branch \nxs{$F$} instead of $A$.

The solution curve \nxs{$F$} in Figure~\ref{fig:solutionsmallbifurcationdiagram} is
characterized by a single vortex in the middle of the domain and is unstable for
field strength weaker than $\mu\approx 1.18$. This solution branch extends all the
way up to field strength $\mu\approx 2.30$ where it connects to the trivial
zero solution.

In a physical experiment where the magnetic field is first increased and then decreased,
we would expect to observe hysteresis: while increasing, the system would initially follow
branch $A$, switching to \nxs{$F$} at point $1$; conversely, for decreasing $\mu$, we would
pass from branch \nxs{$F$} to $A$, at point $6$. Hysteresis effects such as this one have
been discussed in \cite{aftalion2002bifurcation}, and observed experimentally in many
superconducting systems (see also Figure~\ref{fig:experimentalCascade}).

\begin{figure} 
\centering
\includegraphics{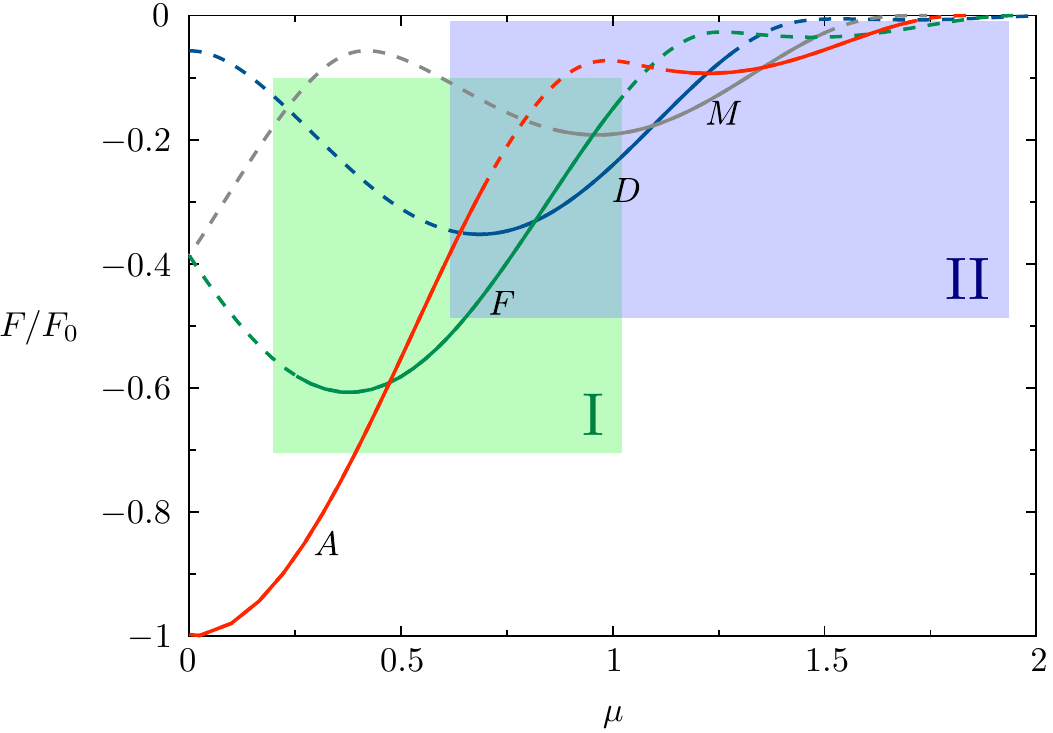}
\caption{The four main branches found for $d=5.5$. The corresponding stable
patterns with one vortex in the middle of the domain are presented in Figure~\ref{fig:example solutions}.
Solid (dashed) lines represent stable (unstable) states.
 Shaded areas are detailed in Figures~\ref{fig:zoneI} and \ref{fig:zoneII}.}
\label{fig:scaling5Cascade}
\end{figure}

\begin{figure}
  \centering
  \includegraphics[height=0.8\textheight]{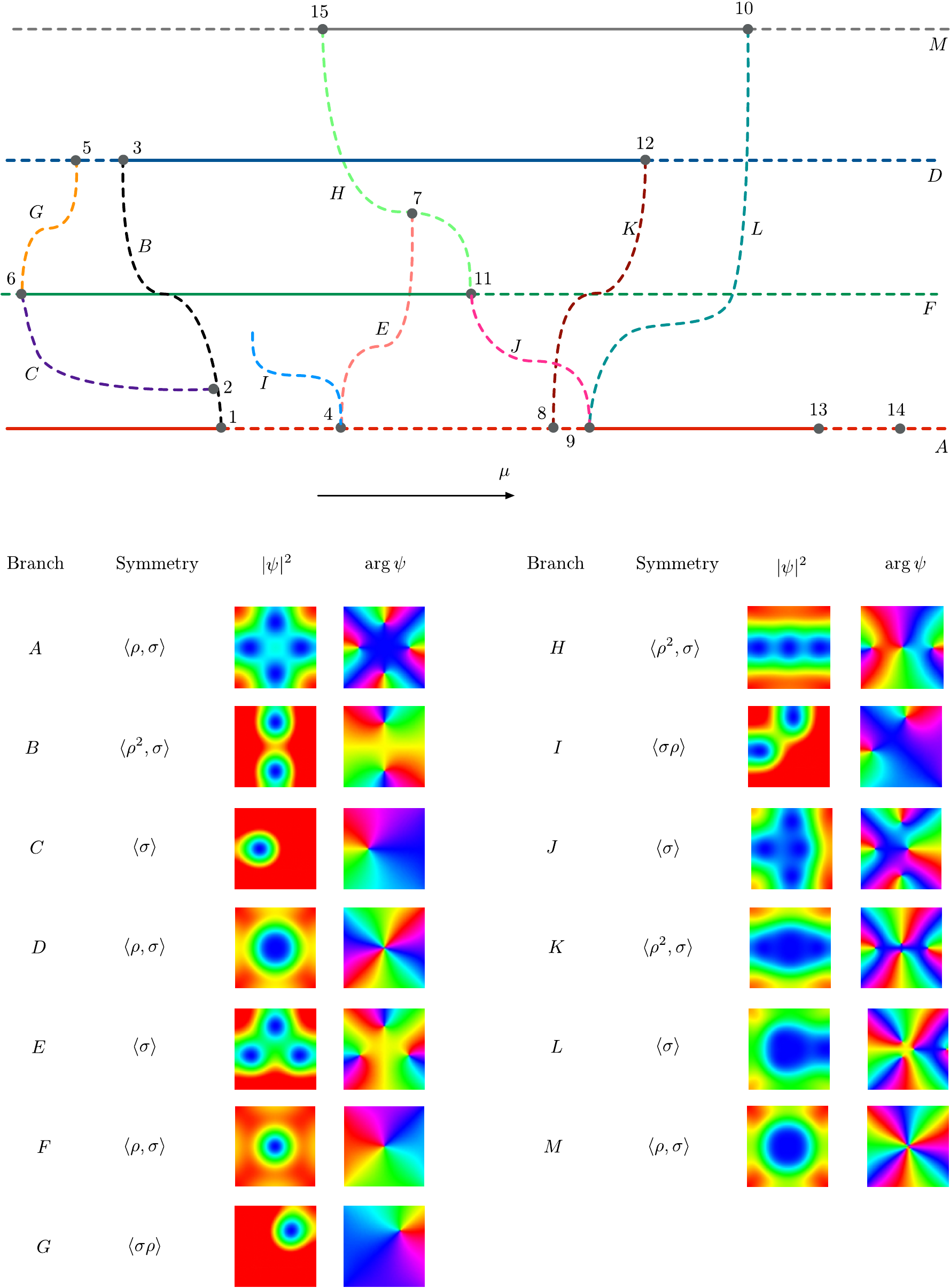}
  \caption{Schematic of solution branches, bifurcations, and patterns found for
  $d=5.5$.}
\label{fig:atlas}
\end{figure}

\subsection{Larger domain size ($d=5.5$)} \label{sec:intermediate}

In this section, we repeat the numerical experiment of Section~\ref{sec:smallsystem}
for a larger sample. In this context, it will be interesting to observe how the
states of branch $A$ destabilize: with edge~length~$d=5.5$, more vortices can
enter the domain, leading to a much more complicated  bifurcation diagram.

Before starting to describe all the branches found by means of numerical continuation, we
anticipate that we found four main branches, as opposed to
the case $d=3.0$, where we computed only two. The four main branches are
collected in Figure~\ref{fig:scaling5Cascade}: they are labeled $F$, $D$, $M$, $A$,
corresponding to states with vorticities $1$, $2$, $3$ and $4$, respectively.
Their stable segments, together with a
few corresponding patterns, have previously been sketched in
Figure~\ref{fig:example solutions}.

In the remainder of this section, we will
concentrate on the two shaded areas (zone I and II) of
Figure~\ref{fig:scaling5Cascade}. In these regions, a series of
symmetry-breaking bifurcations and cross-connecting branches are found.

As in the previous section, we start from the trivial homogeneous state $\psi\equiv 1$
  at $\mu=0$, and increase $\mu$. The resulting solution branch, enjoying full
\Dfour{} symmetry, is labeled $A$ and features four vortices entering the domain
from the sides, similarly to what happens for $d=3.0$. While this scenario resembles
the one described in Section~\ref{sec:smallsystem}, the bifurcations occurring in
zones I and II are quite different from the small-sized case, and we discuss them
one by one
in the remainder of this section. 
We refer the reader to the schematic in Figure~\ref{fig:atlas}, where we present
all the branches, bifurcations, and representative patterns computed for $d=5.5$.

\begin{figure}
\centering
\includegraphics[height=0.8\textheight]{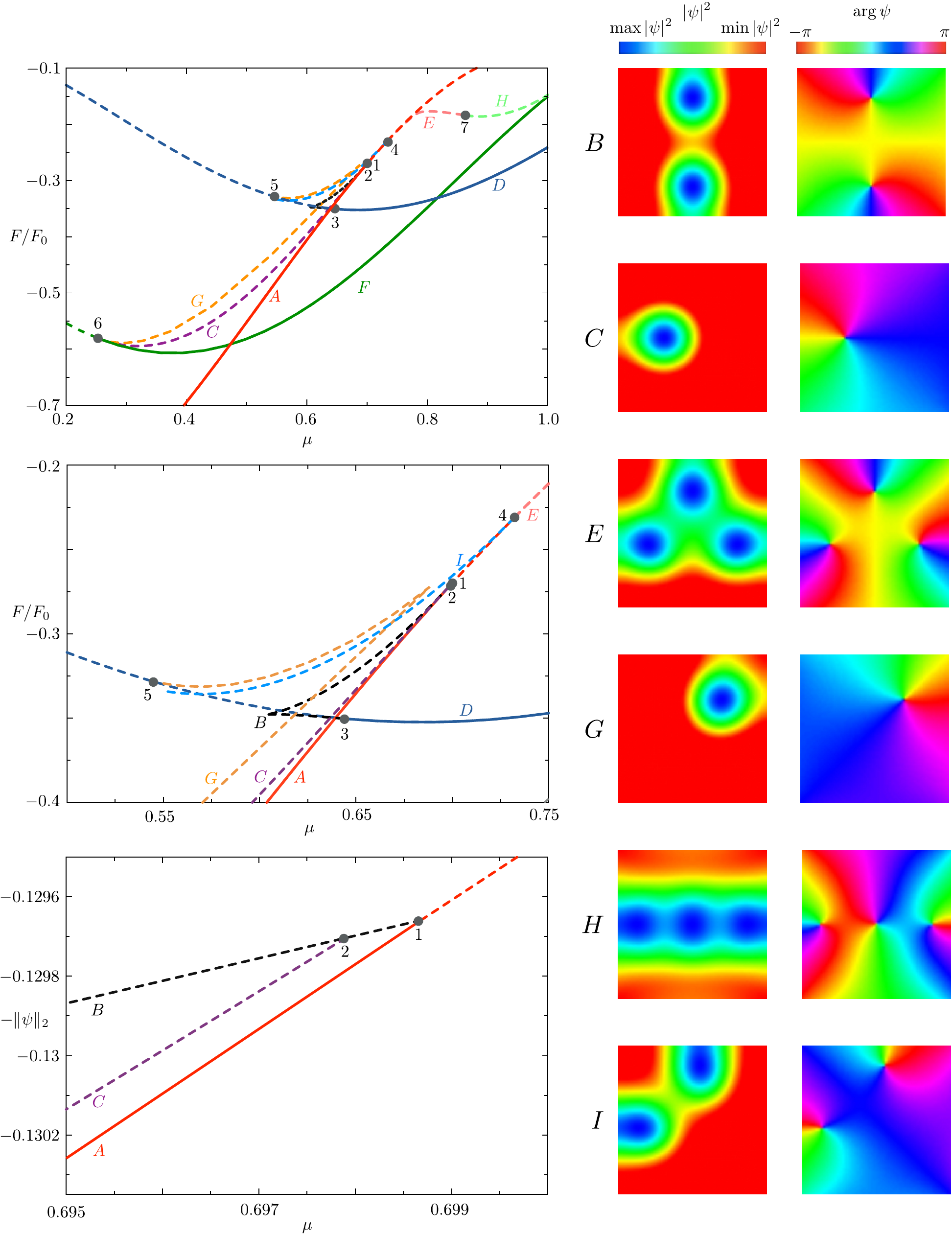}
\caption{Bifurcation diagrams and representative patterns found in zone I of
Figure~\ref{fig:scaling5Cascade}, in the case $d=5.5$. Top and middle panels: free
energy versus magnetic field intensity. Bottom panel: the negative norm of the
solution is used in the bifurcation diagram, in order to separate points $1$ and
$2$. Solid (dashed) lines represent stable (unstable) states.}
\label{fig:zoneI}
\end{figure}
\subsubsection{Zone I}\label{sec:zoneI} Branch $A$ in zone I destabilizes
with a \emph{simple} eigenvalue, at field strength $\mu \approx
0.70$ (see point $1$ in Figure~\ref{fig:zoneI}).  This mechanism is different from
what we found the small-sized system, where an eigenvalue with
multiplicity $2$ becomes unstable.  %
We can still apply the equivariant branching lemma: we expect a single family of
solutions bifurcating from point $1$, corresponding to a one-dimensional
irreducible representation of \Dfour{} \cite{Hoyle:2006:PF}.

The corresponding branch is labeled $B$ in Figure~\ref{fig:zoneI}. It belongs to the conjugacy
class of the isotropy subgroup $\langle \rho^2, \sigma \rangle$, representing the
mirror symmetries along horizontal and vertical center lines. When we follow this branch for
decreasing values of $\mu$, two vortices move simultaneously into the domain from
opposite edges (left-right or top-bottom).

Along branch $B$, we find another symmetry-breaking bifurcation, point $2$, where a
second simple eigenvalue becomes unstable. This is shown in detail in the bottom
panel of Figure~\ref{fig:zoneI}, where we plot the negative value of the
$L_2(\Omega_h)$-norm in order to
visualize the branches better. Branch $C$, emerging from point $2$, has further
reduced symmetry, corresponding to the conjugacy class $\langle \sigma \rangle$,
that is, a family of branches with a single vortex on one of the center lines, away
from the center.

On branch $C$, the vortex moves towards the middle of the sample and is connected via
point $6$, at $\mu\approx 0.25$, to branch $F$, the second main branch with full
\Dfour{} symmetry.  A single vortex sits in the center of the domain throughout
branch $F$ and solutions on $F$ are unstable for fields weaker than $\mu\approx
0.25$. This branch is similar to branch $F$ in the small system described in the
previous section.

Bifurcation point $6$ features a null eigenvalue with multiplicity $2$ and has the
same symmetry properties as the bifurcation points discussed in
Section~\ref{sec:smallsystem}. There, eigenvalues with multiplicity $2$ became
unstable on a branch with \Dfour{} symmetry and two branches emerged with with
symmetries $\langle\sigma\rangle$ and $\langle \sigma\rho \rangle$ (see also
Figure~\ref{fig:solutionsmallbifurcationdiagram}). In the current system, it has
already been found that branch $C$ with symmetry $\langle \sigma \rangle$ connects to
point $6$, and a second branch with symmetry $\langle \sigma\rho \rangle$ is to be
expected.  This branch has a single vortex on one of the diagonals and is shown as
curve $G$ in Figure~\ref{fig:zoneI}. In contrast to the small size system, this curve
does \emph{not} connect to bifurcation point~$1$. Instead, it connects to curve $D$
via bifurcation point $5$.

A branch for which there is no equivalent in the smaller system is branch~$D$
in Figure~\ref{fig:zoneI}, with a single vortex with multiplicity two (and hence
phase change of $2\times2\pi$, a so-called \emph{giant vortex}), in the middle of
the domain. Branch $D$ has full \Dfour{} symmetry and is only stable for fields
larger than $\mu\approx 0.64$. The corresponding bifurcation is marked by point $3$
in Figure~\ref{fig:zoneI} and connects to branch $B$ (see above).

At point $3$, the two vortices of $B$ merge into the giant vortex; similarly,
branch $G$, which emerges from point $6$ on branch $F$, connects to branch point $5$ on
branch $D$.

In the remaining part of zone I, we found that the main branch $A$ has another
instability, at bifurcation point $4$. This bifurcation features a critical
eigenvalue with multiplicity $2$ and thus two families of solution branches emerge.
Along branch $E$, three vortices enter the domain from three of the four sides of the
domain. This branch corresponds to the conjugacy class of the subgroup $\langle
\sigma \rangle$. The three vortices move towards the center of the system along the
branch where they finally merge into a giant vortex with multiplicity $3$ at point
$7$, connecting to branch $H$.

To conclude our exploration of zone I, we examined branch $I$, emerging from point~$4$
on the main branch $A$, for decreasing values of $\mu$.  Patterns on this branch have
two vortices entering from two adjacent edges of the system. This branch is symmetric
under reflections over one of the diagonals and corresponds to the conjugacy class of
the subgroup $\langle \sigma \rho \rangle$.

\begin{figure}
\centering
\includegraphics[height=0.7\textheight]{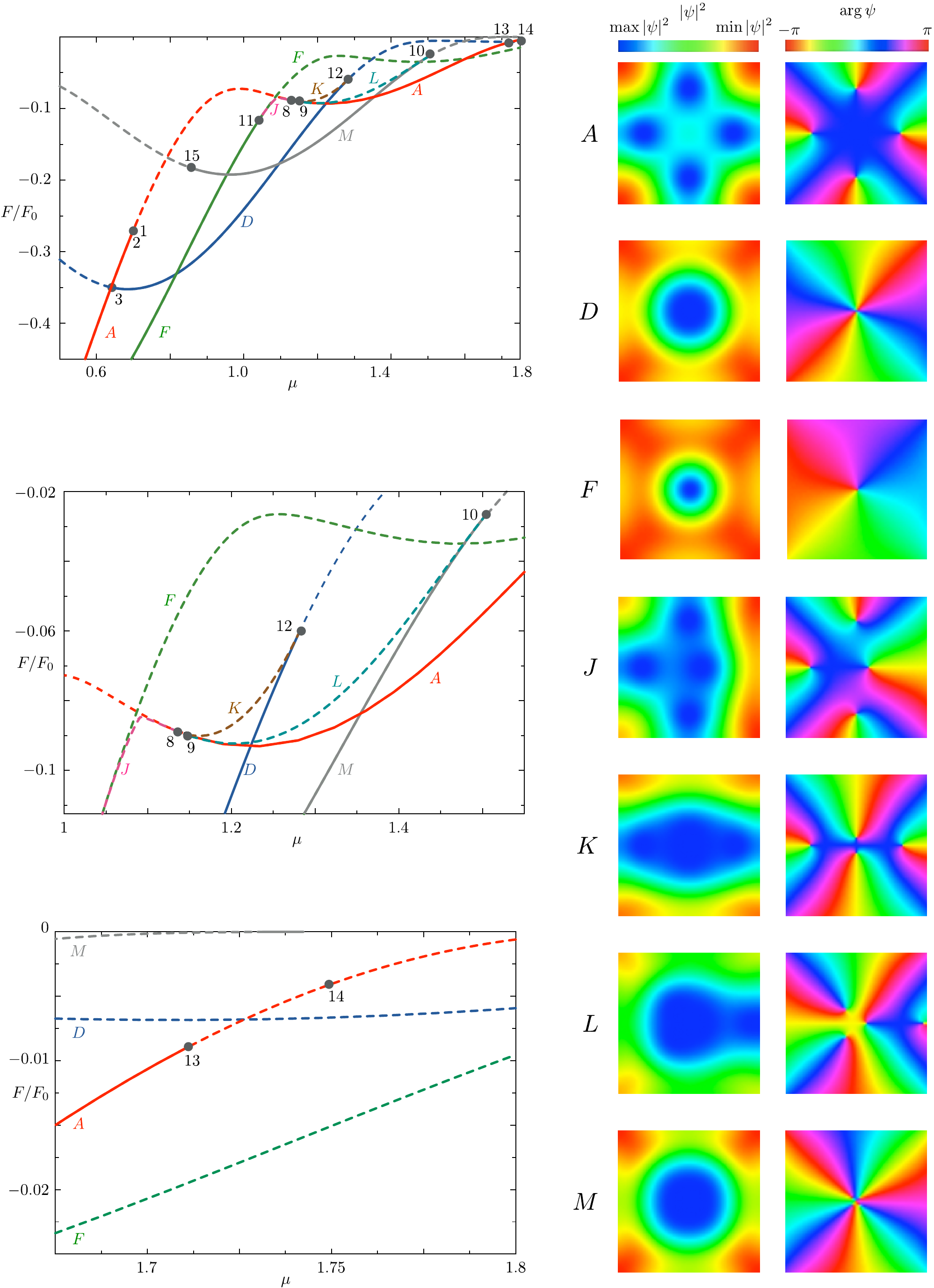}
\caption{Bifurcation diagrams and representative patterns found in zone II of
Figure~\ref{fig:scaling5Cascade}, in the case $d=5.5$. Solid (dashed) lines represent
stable (unstable) states.}
\label{fig:zoneII}
\end{figure}

\subsubsection{Zone II} \label{sec:zoneII} We now move to the upper part of the
bifurcation diagram in Figure~\ref{fig:solutionsmallbifurcationdiagram}.
An important difference from the small-sized system
$d=3$ is that the main branch $A$ restabilizes as the field increases, as shown
Figure~\ref{fig:zoneII}.

As we increase $\mu$ along the main branch $A$, four vortices are moving in from the
midpoints of the edges towards the center; the solutions maintain full \Dfour{}
symmetry. At field strength $\mu\approx 1.07$, the four vortices arrive at
the center and form a giant vortex with multiplicity $4$.
As the field strengthens further, this giant
vortex breaks up again and four separate vortices move away from the center
along the \emph{diagonals}. Note that there is no bifurcation point associated
with this reorganization as none of the eigenvalues of the Jacobian crosses the
origin.

At field strength $\mu\approx 1.14$, one of the unstable eigenvalues of bifurcation
point $1$ restabilizes. This yields bifurcation point $8$ in
Figure~\ref{fig:zoneII}.
From $8$, branch $K$ emerges and connects to branch $D$, with a vortex
of multiplicity $2$ in the center of the domain. Along branch $K$, two of the four
vortices are pushed out of the sample along one of the center lines, while the two remaining reorganize
into a giant vortex of multiplicity $2$ (point 12). A sequence of patterns of branch $K$ can
be found in Figure~\ref{fig:ksequence}.

Branch $A$ restabilizes at field strength $\mu\approx 1.15$. The pattern with four
symmetric vortices on the diagonals is now stable. The bifurcation point that marks
this transition is labeled as point $9$ in Figure~\ref{fig:zoneII}. 
Two solution curves emerge from point $9$, namely branches $L$ and $J$.

Branch $L$, connecting to branch $M$ via point $10$, features five vortices, as can
be seen in Figure~\ref{fig:lsequence}: four vortices arranged symmetrically, rather
close to the center, and a single antivortex at the center of the domain, so
that the total vorticity of the configuration is $3$. A giant vortex of multiplicity
$3$ is formed at bifurcation point $10$ on branch $M$, where it is
unstable. The fact that the vortices do not arrange as a giant vortex with vorticity
$3$ in a stable fashion has been predicted in~\cite{chibotaru2000symmetry}. Solutions
on $M$ are unstable for weak fields strengths (see bifurcation $15$ in
Figure~\ref{fig:zoneII}).

In a similar way, branch $J$ starts at point 9 and connects to point 11 on branch $F$
for decreasing $\mu$. The patterns along this branch are shown in
the sequence of snapshots in Figure~\ref{fig:jsequence}.

At field strength $\mu\approx 1.50$, the main branch $A$ loses its
stability again at point $13$ in a scenario similar to the small-sized system discussed in
Section~\ref{sec:smallsystem}. The eigenvalues of the Jacobian at this bifurcation point are
degenerate and two branches emerge, each of which has a single vortex
entering either along the diagonals or along the center lines. These
branches connect to a stable branch with five vortices
organized like the five dots on a dice.
This branch has been omitted in the figures. Further on the
main branch, a second simple eigenvalue becomes unstable at point $14$.

\begin{figure}
  \centering
  \includegraphics[width=1.0\textwidth]{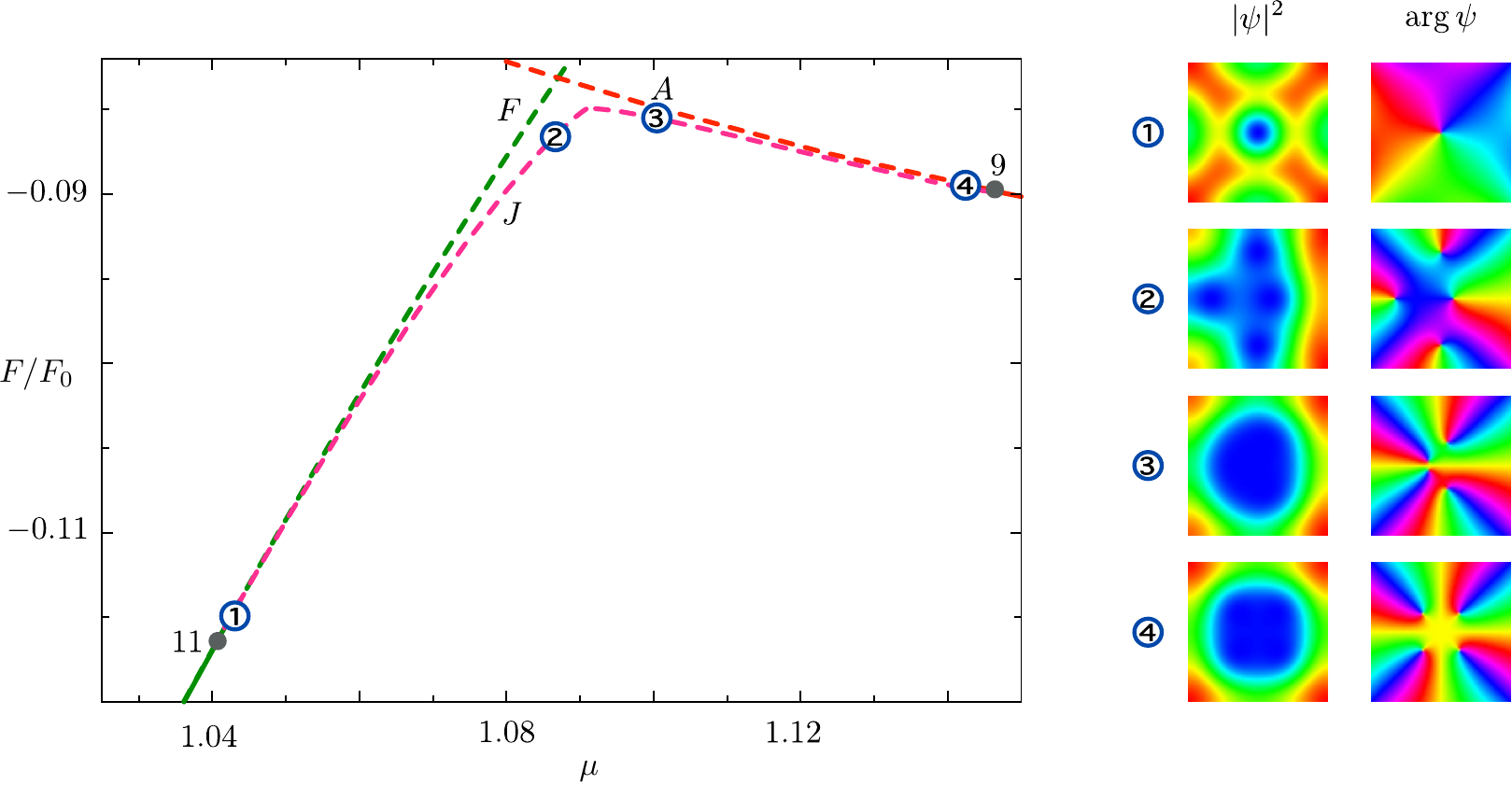}
  \caption{Patterns of branch $J$ in zone II (see Figure~\ref{fig:zoneII}). 
Branch $J$ bifurcates from branch $F$, which  has a single vortex with multiplicity
1,
  and connects to branch $A$ at bifurcation
  point $9$.}
  \label{fig:jsequence}
\end{figure}

\begin{figure}
  \centering
  \includegraphics[width=1.0\textwidth]{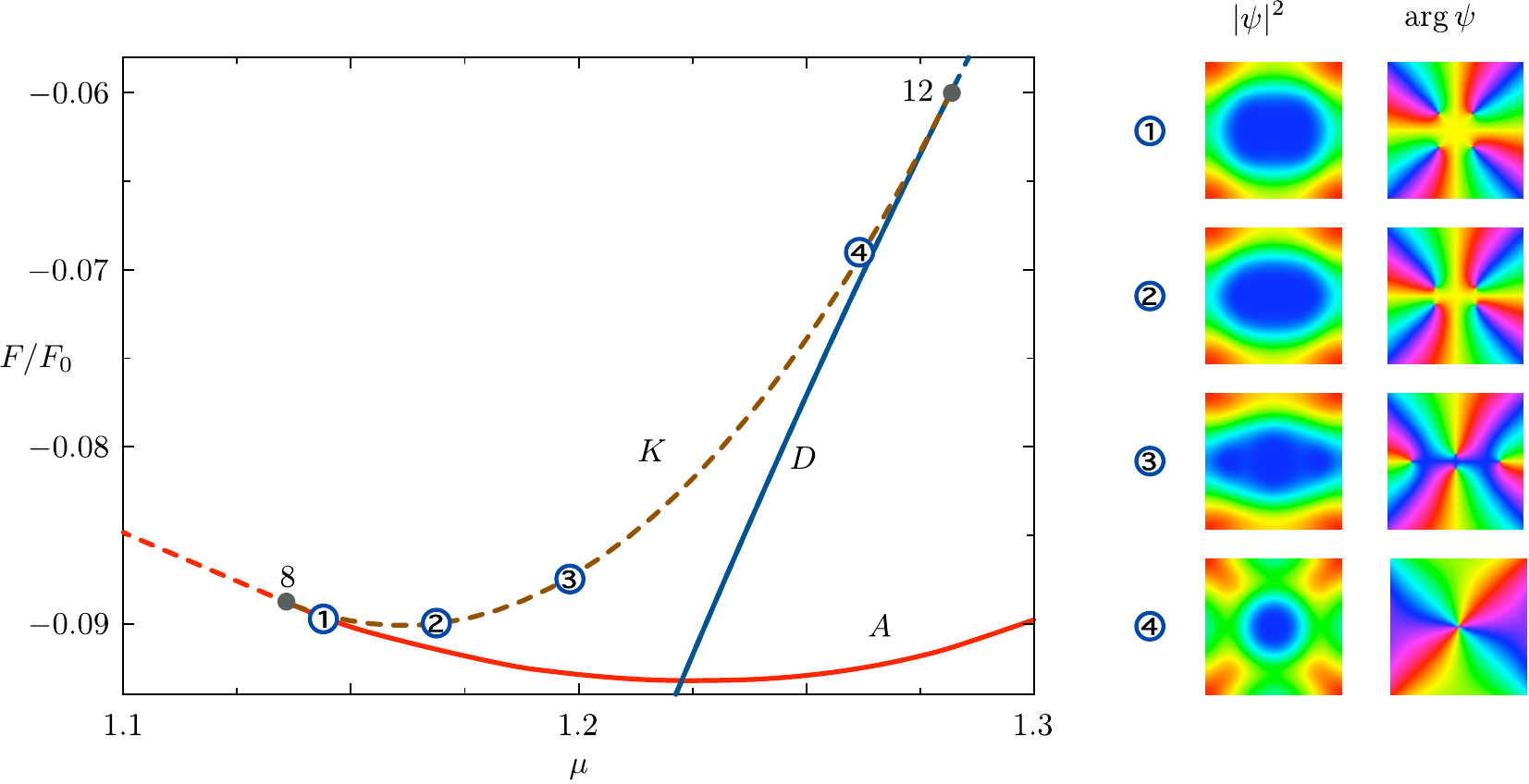}
  \caption{Patterns on branch $K$ in zone II (see Figure~\ref{fig:zoneII}). The
  branch bifurcates off from branch $A$, with four vortices, at bifurcation point
  $8$. Two of the four vortices are pushed out of the sample, while the remaining two
  reorganize into a giant vortex of multiplicity 2 at point 12.}
  \label{fig:ksequence}
\end{figure}

\begin{figure}
  \centering
  \includegraphics[width=1.0\textwidth]{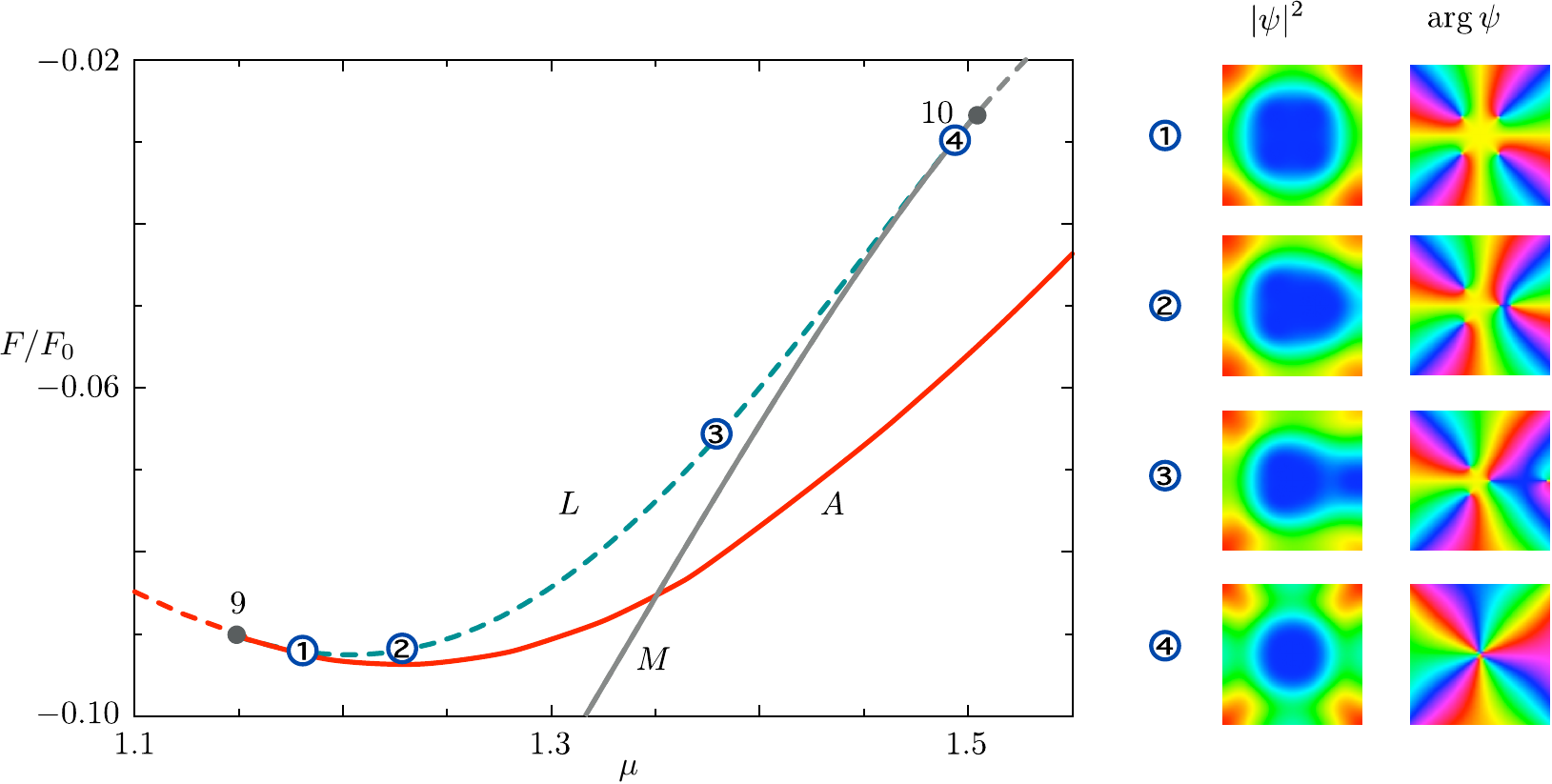}
  \caption{Patterns on branch $L$ in zone II (see Figure~\ref{fig:zoneII}). The
  patterns show five vortices: four vortices are arranged symmetrically, rather close
  to the center, and a single antivortex sits at the very center of the domain, so
  that the total vorticity of the configuration is 3. At bifurcation point $10$, when
  a giant vortex of multiplicity $3$ is formed, the pattern becomes unstable, on
  branch $M$.}
  \label{fig:lsequence}
\end{figure}

\section{Discussion and conclusions}
We have presented an initial exploration of the symmetry-breaking
bifurcations of the vortex patterns as modeled by the Ginzburg--Landau equations.
In the case of extreme type-II superconductors, we assumed a homogeneous applied
magnetic field and showed how the vortices reorganize as the
strength of the applied field is varied. In the small square domain ($d=3$), we
believe to have
given a complete account of the instabilities of the system. For a larger system, the
bifurcation diagram becomes much more complicated, and we found a large number
of states and symmetry-breaking bifurcations.

The paper also presents a study of the symmetries of the system.
It has been shown that the continuous system bears symmetries isomorphic to
$\T\times$\Dfour.
The discretization has been chosen in such a way that it
preserves to machine accuracy both phase and geometric symmetries. 

Owing to the symmetries of the system, it is possible to use the
Equivariant Branching Lemma in order to predict the existence of new branches at
symmetry-breaking bifurcations, and subsequently compute them numerically. To the best
of the authors' knowledge, most of the patterns contained in this paper are unknown to the
physics community: even though unstable patterns can not be obtained experimentally,
we point out that the methodology proposed in this context could be effectively used
to find new stable patterns.

The present paper analyzes the Ginzburg--Landau system on a square, but
the same technique can be applied to all geometries with inherent symmetries, e.g.,
regular $n$-gons.
It is not immediately obvious, though, how to choose the
magnetic vector potential gauge such that the corresponding Ginzburg--Landau
formulation remains equivariant with respect to $D_n$; some work in this area
has been done in~\cite{CC:2005:GLD}. Note that, for increasing $n$, the ever more complicated
subgroup structure of $D_n$ will lead to different bifurcation scenarios
\cite{golubitsky1988singularities,golubitsky2002symmetry}.

In the present paper we simplified the Ginzburg--Landau equations considering the
large-$\kappa$ limit, where the
equation for the magnetic vector potential $\A$ decouples from the order parameter $\psi$.
It will be necessary, in the future, to study the
bifurcations in the coupled system for intermediate and small
values of $\kappa$.
However, this task will also pose new numerical challenges: the magnetic vector
potential appears as an additional (vector-valued) unknown and its domain of definition is the whole space.
In practice, the vector potential will approach its boundary condition defined by
$\H_0$ sufficiently far away from the sample, but the validity of this approximation
is still an open problem.
The coupled system will in any case hold many more
unknowns, and a robust preconditioning strategy for solving the appearing
Jacobian systems will be crucial.
However, the regularization technique that we employed for the extreme type-II
case is applicable for finite values of $\kappa$ and for generic spatial
discretizations of the Ginzburg--Landau problem.

Nevertheless, we believe that results of this paper
are a first step in understanding the
bifurcations in the coupled Ginzburg--Landau system for various mesoscopic
systems that are relevant for nanoscale devices. The approach proposed here opens up
the possibility of a systematic exploration of the solution landscape in regions
that are precluded to direct numerical simulation.

\section*{Acknowledgements}
We acknowledge fruitful discussions with Golibjon Berdiyorov,
Milorad Milo\v{s}evi\'c,  Ben Xu, Bart Partoens,
Andrew G. Salinger, Eric T. Phipps, Mathieu Desroches, Rebecca Hoyle, and Philip
Aston. We are also grateful to FWO-Vlaanderen for financial support through
the project
G017408N. Daniele Avitabile acknowledges EPSRC for funding his research with the
grant EP/E032249/1.

\appendix

\section{Extension of Keller's bordering lemma}

Keller's bordering lemma~\cite{Keller:1976:NSB} provides conditions on how
a finite-dimensional linear system with a singularity of dimension $1$
can be regularized by adding an additional unknown as well as an additional equation.
In the present context, however, it is necessary to formulate the lemma in
general vector spaces. Also, the defect of the present problem may be greater
than one. Such situations occur, for example, in several branch points described
in section~\ref{sec:results}. The following lemma shows that it is always possible to remove one
of the singularities.

\begin{lemma}\label{lemma:keller}
Let $X$, $Y$ be $\K$-vector spaces and let $\L:X\to Y$ linear with
$\dim\ker\L = k>0$. Let further $b\in Y$, $d\in\K$, and $f:X\to\K$ a
linear functional. Let the operator $\widetilde{\L}: X\times\K \to
Y\times\K$ be defined by
\[
\widetilde{\L}
\tilde{x}
\dfn
\begin{pmatrix}
\L x + b \xi \\
f(x) + d \xi
\end{pmatrix}
\]
for all $\tilde{x} = (x,\xi)^\tp \in X\times\K$.
Then $\tilde{k}\dfn\dim\ker\widetilde{\L}<k$ if and only if $b\notin\range(\L)$
and there exists a $v\in\ker\L$ with $f(v)\neq0$.
\end{lemma}
\begin{proof}
On the one hand, let  $b\notin\range(\L)$ and let $v\in\ker\L$ with $f(v)\neq 0$.
Let $\{(w^{(i)},\xi_i)^\tp\}_{i=1}^{\tilde{k}}\subset X\times\K$ denote a basis of $\ker\widetilde{\L}$, and take a $\tilde{x}\in\ker\widetilde{\L}$,
\[
\tilde{x}
= \sum\limits_{i=1}^{\tilde{k}} \alpha_i
\begin{pmatrix}
w^{(i)}\\
\xi_i
\end{pmatrix}
\]
with arbitrary $\alpha_i\in\K$. With this representation, we have
\[
\begin{split}
0 &= \L  \sum\limits_{i=1}^{\tilde{k}} \alpha_i w^{(i)} + b  \sum\limits_{i=1}^{\tilde{k}} \alpha_i \xi_i,\\
0 &= f\left(\sum\limits_{i=1}^{\tilde{k}} \alpha_i w^{(i)}\right) + d \sum\limits_{i=1}^{\tilde{k}} \alpha_i \xi_i.
\end{split}
\]
Because $b\notin\range(\L)$, it must be $\sum_{i=1}^{\tilde{k}} \alpha_i \xi_i=0$
as otherwise
\[
b =  \left(\sum_{i=1}^{\tilde{k}} \alpha_i \xi_i\right)^{-1} \sum_{i=1}^{\tilde{k}} \alpha_i \L w^{(i)}\in\range(\L).
\]
Because the $\alpha_i$ are arbitrary, we have $\xi_i=0$ for all
$i$. Since $\{(w^{(i)},\xi_i)^\tp\}_{i=1}^{\tilde{k}}$ is linearly independent
in $X \times \mathbb{K}$ and all $\xi_i$ are zero,
$\{w^{(i)}\}_{i=1}^{\tilde{k}}$ is linearly independent in $X$. Besides that,
it follows that $\sum_{i=1}^{\tilde{k}} \alpha_i
w^{(i)}\in\ker\L$, and again because the $\alpha_i$ are arbitrary, we
have $w^{(i)}\in\ker\L$ for all $i\in\{1,\dots,\tilde{k}\}$. Hence
$\dim\ker\L\ge \tilde{k}$.   One can exclude  $\tilde{k}=\dim\ker\L$ since then
$\ker\L=\spn\{w^{(i)}\}_{i=1}^{\tilde{k}}$, and at the same time $0 =
f(\sum_{i=1}^k \alpha_i w^{(i)})$ for arbitrary $\alpha_i$.  This
 contradicts  the assumption there is a
$v\in\ker\L$ with $f(v)\neq 0$. Hence $\tilde{k}<k$.

On the other hand, let $\tilde{k}<k$. Consider the set
$W\dfn \ker\L\times\{0\}$. Obviously it is $\dim W=k$, and additionally
for any $\widetilde{w} = (w,0)^\tp \in W$, one has
\[
\widetilde{L} \widetilde{w}
=
\begin{pmatrix}
\L( w ) + 0\cdot b\\
f(w) + 0\cdot d
\end{pmatrix}
=
\begin{pmatrix}
0\\
f(w)
\end{pmatrix}.
\]
Hence, there must be a $v\in\ker\L$ with $f(v)\neq 0$ as
as otherwise $W\subseteq \ker\L$ and $\tilde{k}\ge k$.

It remains to be shown that $b\notin\range(\L)$, and we will do this by
contradiction:
Suppose that $b\in\range(\L)$ with a $p\in X$
such that $b=\L p$. Note that for any given $\alpha\in\R$, it is also
$b=\L(p+\alpha v)$, where $v\in\ker\L$ such that $f(v)\neq 0$.
Choose $\alpha$ such that $\alpha \neq (d-f(p))/f(v)$ and
let $\hat{p}\dfn p+\alpha v$ and
$S\dfn\{(w^{(i)}-\xi_i \hat{p},\xi_i)^\tp\}_{i=1}^k$
with $\xi_i\dfn f(w^{(i)})/(f(\hat{p})-d)$.
It can be checked that $S$ is linearly independent by taking
arbitrary $\{\beta_i\}_{i=1}^k\subset\R$ and demanding
\[
0 \stackrel{!}{=}
\sum_{i=1}^k \beta_i
\begin{pmatrix}
w^{(i)}-\xi_i \hat{p}\\
\xi_i
\end{pmatrix}.
\]
The second component yields $0=\sum_{i=1}^k \beta_i\xi_i$, which results in
\[
0 \stackrel{!}{=}
\sum_{i=1}^k \beta_i
\begin{pmatrix}
w^{(i)}\\
0
\end{pmatrix}
+
\sum_{i=1}^k \beta_i\xi_i
\begin{pmatrix}
- \hat{p}\\
1
\end{pmatrix}
=
\sum_{i=1}^k \beta_i
\begin{pmatrix}
w^{(i)}\\
0
\end{pmatrix}
\]
The set $\{w^{(i)}\}_{i=1}^k$ is, however, linearly independent such that all $\beta_i$ must vanish. Hence
$S$ is linearly independent.
But $S$ is also a subset of $\ker\widetilde{\L}$ as
\[
\widetilde{\L}
\begin{pmatrix}
w^{(i)}-\xi_i \hat{p}\\
\xi_i
\end{pmatrix}
=
\begin{pmatrix}
\L( w^{(i)}-\xi_i \hat{p}) + b \xi_i\\
f(w^{(i)}-\xi_i \hat{p}) + d \xi_i
\end{pmatrix}
=
\begin{pmatrix}
\L w^{(i)}-\xi_i \L \hat{p} + b \xi_i\\
f(w^{(i)})-\xi_i f(\hat{p}) + d \xi_i
\end{pmatrix}
=
\begin{pmatrix}
0\\
0
\end{pmatrix}
\]
This means that $\tilde{k}\ge k$, which is a
contradiction.

\end{proof}

\bibliographystyle{plain}
\bibliography{bibtex/gl}

\begin{thebibliography}{10}

\bibitem{abrikosov1957}
A.A. Abrikosov.
\newblock Magnetic properties of superconductors of the second group.
\newblock {\em Sov. Phys. JETP}, 5:1174, 1957.

\bibitem{aftalion2000asymptotic}
A.~Aftalion and S.J. Chapman.
\newblock Asymptotic analysis of a secondary bifurcation of the one-dimensional
  {Ginzburg-Landau} equations of superconductivity.
\newblock {\em SIAM Journal on Applied Mathematics}, 60(4):1157--1176, 2000.

\bibitem{aftalion2002bifurcation}
A.~Aftalion and Q.~Du.
\newblock The bifurcation diagrams for the {Ginzburg-Landau} system of
  superconductivity.
\newblock {\em Physica D: Nonlinear Phenomena}, 163(1-2):94--105, 2002.

\bibitem{aftaliontroy}
A.~Aftalion and W.C. Tray.
\newblock One the solutions of the the one-dimensional {Ginzburg-Landau}
  equations for superconductivity.
\newblock {\em Physica D}, 132:214--232, 1999.

\bibitem{aladyshkin2009nucleation}
A.Y. Aladyshkin, A.V. Silhanek, W.~Gillijns, and V.V. Moshchalkov.
\newblock Nucleation of superconductivity and vortex matter in
  superconductor--ferromagnet hybrids.
\newblock {\em Superconductor Science and Technology}, 22, 2009.

\bibitem{aranson2002}
I.S. Aranson and L.~Kramer.
\newblock The world of the complex {Ginzburg-Landau} equation.
\newblock {\em Reviews of Modern Physics}, 74(1):99--143, 2002.

\bibitem{Avron1978}
J.~Avron, I.~Herbst, and B.~Simon.
\newblock {Schr\"odinger operators with magnetic fields. I. General
  Interactions}.
\newblock {\em Duke Mathematical Journal}, 45(4):847--883, 1978.

\bibitem{BBX:2003:REE}
C.~Bacuta, J.H. Bramble, and J.~Xu.
\newblock Regularity estimates for elliptic boundary value problems with smooth
  data on polygonal domains.
\newblock {\em Journal of Numerical Mathematics}, 11(2):75--94, June 2003.

\bibitem{PhysRevB.65.104515}
B.J. Baelus and F.M. Peeters.
\newblock Dependence of the vortex configuration on the geometry of mesoscopic
  flat samples.
\newblock {\em Phys. Rev. B}, 65(10):104515, February 2002.

\bibitem{bethuel1994ginzburg}
F.~Bethuel, H.~Brezis, and F.~H\'elein.
\newblock {\em {Ginzburg-Landau} Vortices}.
\newblock Springer, 1994.

\bibitem{beyn2004}
W.J. Beyn and V.~Th{\"u}mmler.
\newblock Freezing solutions of equivariant evolution equations.
\newblock {\em SIAM Journal on Applied Dynamical Systems}, 3(2):85--116, 2004.

\bibitem{beyn2007}
W.J. Beyn and V.~Th{\"u}mmler.
\newblock {\em Numerical Continuation Methods for Dynamical Systems}, chapter
  Phase conditions, symmetries and PDE continuation, pages 301--330.
\newblock Canopus, Springer, 2007.

\bibitem{PhysRevB.70.144523}
L.R.E. Cabral, B.J. Baelus, and F.M. Peeters.
\newblock From vortex molecules to the {Abrikosov} lattice in thin mesoscopic
  superconducting disks.
\newblock {\em Phys. Rev. B}, 70(14), October 2004.

\bibitem{CS:2007:NCC}
A.R. Champneys and B.~Sandstede.
\newblock {\em Numerical Continuation Methods for Dynamical Systems}, chapter
  Numerical computation of coherent structures, pages 331--358.
\newblock Canopus, Springer, 2007.

\bibitem{chibotaru2000symmetry}
L.F. Chibotaru, A.~Ceulemans, V.~Bruyndoncx, and V.V. Moshchalkov.
\newblock Symmetry-induced formation of antivortices in mesoscopic
  superconductors.
\newblock {\em Nature}, 42(4):555--598, 2000.

\bibitem{CC:2005:GLD}
L.F. Chibotaru, A.~Ceulemans, M.~Morelle, G.~Teniers, C.~Carballeira, and V.V.
  Moshchalkov.
\newblock {Ginzburg--Landau} description of confinement and quantization
  effects in mesoscopic superconductors.
\newblock {\em Journal of Mathematical Physics}, 46(9), September 2005.

\bibitem{dancer2000global}
E.N. Dancer and S.P. Hastings.
\newblock On the global bifurcation diagram for the one-dimensional
  {Ginzburg--Landau} model of superconductivity.
\newblock {\em European Journal of Applied Mathematics}, 11(03):271--291, 2000.

\bibitem{demmel}
J.W. Demmel.
\newblock {\em Applied numerical linear algebra}.
\newblock Society for Industrial and Applied Mathematics, 1997.

\bibitem{PhysRevLett.79.4653}
P.~Singha Deo, V.A. Schweigert, F.M. Peeters, and A.K. Geim.
\newblock Magnetization of mesoscopic superconducting disks.
\newblock {\em Phys. Rev. Lett.}, 79(23):4653--4656, December 1997.

\bibitem{Du:1998:DGI}
Q.~Du.
\newblock Discrete gauge invariant approximations of a time dependent
  {Ginzburg--Landau} model of superconductivity.
\newblock {\em Math. Comput.}, 67(223):965--986, 1998.

\bibitem{DGP:1992:AAG}
Q.~Du, M.D. Gunzburger, and J.S. Peterson.
\newblock Analysis and approximation of the {Ginzburg--Landau} model of
  superconductivity.
\newblock {\em SIAM Rev.}, 34:54--81, March 1992.

\bibitem{DJ:2004:NSQ}
Q.~Du and L.~Ju.
\newblock Numerical simulations of the quantized vortices on a thin
  superconducting hollow sphere.
\newblock {\em Journal of Computational Physics}, 201:511--530, 2004.

\bibitem{golubitsky1988singularities}
M.~Golubitsky, D.G. Schaeffer, and I.~Stewart.
\newblock {\em Singularities and groups in bifurcation theory}.
\newblock Springer Verlag, 1988.

\bibitem{golubitsky2002symmetry}
M.~Golubitsky and I.~Stewart.
\newblock {\em The symmetry perspective}.
\newblock Birkh\"auser, 2002.

\bibitem{goodman1966}
B.B. Goodman.
\newblock Type ii superconductors.
\newblock {\em Reports on progress in physics}, 29:445, 1966.

\bibitem{HW:2003:TUG}
M.A. Heroux and J.M. Willenbring.
\newblock {Trilinos Users Guide}.
\newblock Technical Report SAND2003-2952, Sandia National Laboratories, 2003.

\bibitem{Hoyle:2006:PF}
R.~Hoyle.
\newblock {\em Pattern formation}.
\newblock Cambridge University Press, 2006.

\bibitem{KK:1995:VCT}
H.G. Kaper and M.K. Kwong.
\newblock Vortex configurations in type-{II} superconducting films.
\newblock {\em Journal of Computational Physics}, 119(1):120--131, June 1995.

\bibitem{Keller:1976:NSB}
H.B. Keller.
\newblock Numerical solution of bifurcation and nonlinear eigenvalue problems.
\newblock In Paul~H. Rabinowitz, editor, {\em Applications of bifurcation
  theory: proceedings of an advanced seminar}, pages 359--384, University of
  Wisconsin--Madison, October 1976. Academic Press, New York.

\bibitem{Kelley1995}
C.T. Kelley.
\newblock {\em Iterative Methods for Linear and Nonlinear Equations}, volume~16
  of {\em Frontiers in Applied Mathematics}.
\newblock SIAM, 1995.

\bibitem{krauskopf2007}
B.~Krauskopf.
\newblock {\em Numerical Continuation Methods for Dynamical Systems: {Path}
  following and boundary value problems}.
\newblock Springer Verlag, 2007.

\bibitem{LD:1997:GLV}
F.-H. Lin and Q.~Du.
\newblock {Ginzburg--Landau} vortices: dynamics, pinning, and hysteresis.
\newblock {\em SIAM J. Math. Anal.}, 28(6):1265--1293, 1997.

\bibitem{rowley2003}
C.W. Rowley, I.G. Kevrekidis, J.E. Marsden, and K.~Lust.
\newblock Reduction and reconstruction for self-similar dynamical systems.
\newblock {\em Nonlinearity}, 16:1257, 2003.

\bibitem{sandier2007vortices}
E.~Sandier and S.~Serfaty.
\newblock {\em Vortices in the magnetic {Ginzburg-Landau} model}.
\newblock Birkh\"auser, 2007.

\bibitem{PhysRevLett.81.2783}
V.A. Schweigert, F.M. Peeters, and P.~Singha Deo.
\newblock Vortex phase diagram for mesoscopic superconducting disks.
\newblock {\em Phys. Rev. Lett.}, 81(13):2783--2786, September 1998.

\bibitem{Schweitzer1967}
D.G. Schweitzer and M.~Garber.
\newblock Hysteresis in superconductors. {II.} {Experimental} tests for
  critical states.
\newblock {\em Phys. Rev.}, 160(2):348--358, August 1967.

\end{thebibliography}
\end{document}